\documentclass{article}

\usepackage{amsmath,amstext,amsfonts,amssymb,amsthm,bbold,dsfont}
\usepackage{graphicx,subfigure,color}
\usepackage{fourier}
\usepackage[latin1]{inputenc}
\usepackage[T1]{fontenc}
\usepackage[english]{babel}
\usepackage{hyperref}

\def \A {\mathbf{A}}
\def \Acal{\mathcal{A}}

\def \B {\mathbf{B}}

\def \Ccal{\mathcal{C}}
\def \Cbb{\mathbb{C}}
\def \D {\mathbf{D}}
\def \Dcal {\mathcal{D}}

\def \drm {\mathrm{d}}

\def \Ebb {\mathbb{E}}

\def \I {\mathbf{I}}

\def \Ical {\mathcal{I}}

\def \Ocal{\mathcal{O}}

\def \Prm {\mathrm{P}}

\def \Rbb {\mathbb{R}}

\def \s {\mathbf{s}}

\def \Scal {\mathcal{S}}

\def \T {\mathbf{T}}

\def \u {\mathbf{u}}

\def \v {\mathbf{v}}

\def \Vcal {\mathcal{V}}
\def \W {\mathbf{W}}

\def \y {\mathbf{y}}

\def \Sigmabs {\boldsymbol{\Sigma}}

\def \Tr {\mathrm{Tr}\,}

\def \supp{\mathrm{supp}}
\def \Var{\mathrm{Var}}
\def \Re {\mathrm{Re}}
\def \Im {\mathrm{Im}}

\newtheorem{corollary}{Corollary}
\newtheorem{theorem}{Theorem}

\newtheorem{lemma}{Lemma}

\newtheorem{property}{Property}
\renewenvironment{proof}
{\noindent\textbf{Proof}:}{\hfill$\square$\bigbreak} 

\newcounter{countassum}
\newenvironment{assumption}
{
	\refstepcounter{countassum}
	\begin{flushleft}
	\noindent\textbf{Assumption A-\thecountassum}:
	\it
}
{
	\end{flushleft}
	
}

\title
{
	Almost sure localization of the eigenvalues in a Gaussian information plus noise model. Application to the spiked models.
	\footnote{This work was partially supported by the French program ANR-07-MDCO-012-01 "SESAME"}
}
\date{}
\author
{
	Philippe Loubaton, Pascal Vallet 
	\\
	\normalsize{Université Paris-Est - Marne la Vallée} 
	\\
	\normalsize{LIGM (CNRS- UMR 8049)}
	\\
	\normalsize{5 Boulevard Descartes, 77454 Marne-la-Vallée (France)}
	\\
	\normalsize{\textsf{\{loubaton,vallet\}@univ-mlv.fr}}
}

\begin{document}
%\tableofcontents

\maketitle
\abstract
{
        Let $\boldsymbol{\Sigma}_N$ be a $M \times N$ random matrix defined by $\boldsymbol{\Sigma}_N = \B_N + \sigma \W_N$ where $\B_N$ is a uniformly bounded 
        deterministic matrix and where $\W_N$ is an independent identically distributed complex Gaussian matrix with zero mean and variance $\frac{1}{N}$ entries. 
        The purpose of this paper is to study the almost sure location of the eigenvalues $\hat{\lambda}_{1,N} \geq \ldots \geq \hat{\lambda}_{M,N}$ of 
        the Gram matrix ${\boldsymbol \Sigma}_N {\boldsymbol \Sigma}_N^*$ when $M$ and $N$ converge to $+\infty$ such that the ratio $c_N = \frac{M}{N}$ converges 
        towards a constant $c > 0$. 
        The results are used in order to derive, using an alternative approach, known results concerning the behaviour of the largest eigenvalues of 
        ${\boldsymbol \Sigma}_N {\boldsymbol \Sigma}_N^*$ when the rank of $ \B_N$ remains fixed and $M$ and $N$ converge to $+\infty$.
}
\\~\\
Submitted to EJP on September 29, 2010, final version accepted on September 27, 2011.

\newpage
\normalsize

%%%%%%%%%%%%%%%%%%%%%%%%%%%%%%%%%%%%%%%%%%%%%%%%%%%%%%%%%%%%%%%%%%%%%%%%%%%%%%%%%%%%%%%%%%%%%%%%%%%%%%%%%%%%%%%%%%%%%%%%%%%%%%%%%%%%%%%%%%%%%%%%%%%%%%%%%%%%%%%%%%%
%
%
%
%
%	INTRODUCTION
%
%
%
%
%
%%%%%%%%%%%%%%%%%%%%%%%%%%%%%%%%%%%%%%%%%%%%%%%%%%%%%%%%%%%%%%%%%%%%%%%%%%%%%%%%%%%%%%%%%%%%%%%%%%%%%%%%%%%%%%%%%%%%%%%%%%%%%%%%%%%%%%%%%%%%%%%%%%%%%%%%%%%%%%%%%%%

\section{Introduction}

	\paragraph{The addressed problem and the results}

Let $\boldsymbol{\Sigma}_N$ be a $M \times N$ complex-valued matrix defined by 
\begin{align}
	\boldsymbol{\Sigma}_N = \B_N + \sigma \W_N
	\label{eq:model-infonoise}
\end{align}
where $\B_N$ is a $M \times N$ deterministic matrix such that $\sup_N \|\B_N\| < +\infty$, and where $\W_N = [\W_N]_{i,j}$ is a $M \times N$ complex Gaussian random matrix 
with independent identically distributed (i.i.d) entries such that 
$\mathbb{E}\left[[\W_N]_{i,j}\right]= 0$, $\left|\mathbb{E}\left[[\W_N]_{i,j}\right]\right|^{2} = \frac{1}{N}$, $\mathrm{Re}\left([\W_N]_{i,j}\right)$ and  
$\mathrm{Im}\left([\W_N]_{i,j}\right)$ are i.i.d zero mean real Gaussian random variables. 
Model \eqref{eq:model-infonoise} is referred in the literature to as the information plus noise model (see e.g Dozier-Silverstein \cite{dozier2007empirical}). 
In this paper, we assume that $\mathrm{Rank}(\B_N) = K(N)= K < M$ because this assumption is verified in a number of practical situations, in particular in the context of
the spiked models addressed here. 

The purpose of this paper is to study the almost sure location of the eigenvalues $\hat{\lambda}_{1,N} \geq \ldots \geq \hat{\lambda}_{M,N}$
of the Gram matrix $\Sigmabs_N \Sigmabs_N^*$ when $M$ and $N$ converge to $+\infty$ such that the ratio $c_N = \frac{M}{N}$ converges towards a constant $c > 0$ and to take benefit 
of the results to obtain, using a different approach than Benaych-Nadakuditi \cite{benaych2011singular}, the behaviour of the largest eigenvalues of the information plus noise 
spiked models for which the rank $K$ of $\B_N$ remains constant when $M$ and $N$ increase to $+\infty$.

The empirical spectral measure (or eigenvalue distribution) $\hat{\mu}_N = \frac{1}{M} \sum_{m=1}^{M} \delta_{\hat{\lambda}_{m,N}}$ of matrix $\Sigmabs_N \Sigmabs_N^*$ has the same 
asymptotic behaviour than a deterministic probability distribution $\mu_N$ (see e.g. Dozier-Silverstein \cite[Th.1.1]{dozier2007empirical} or Girko \cite[Th.7.4]{girko2001canonical}) 
whose support $\Scal_N$ is the union of disjoint compact intervals called in the following the clusters of $\Scal_N$.
The boundary points of each cluster coincide with the positive extrema of a certain rational function depending on the empirical spectral measure of matrix $\B_N \B_N^*$, 
$\sigma^{2}$ and on the ratio $c_N = \frac{M}{N}$ (see \cite{vallet2010sub}, Thereom 2). 
Each cluster $\Ical$ of $\Scal_N$ appears to be naturally associated to another interval containing a group of consecutive eigenvalues of ${\bf B}_N {\bf B}_N^*$ (\cite{vallet2010sub}). 
It is shown in \cite{vallet2010sub} that the property proved in Bai-Silverstein \cite{bai1998no} holds in the context of model \eqref{eq:model-infonoise}. 
Roughly speaking, it means that for an interval $[a,b]$ located outside $\mathcal{S}_N$ for $N$ large enough, no eigenvalue of 
$\boldsymbol{\Sigma}_N \boldsymbol{\Sigma}_N^*$ belong to $[a,b]$ almost surely, for all large $N$.

In this paper, we establish the analog of the property called in Bai-Silverstein \cite{bai1999exact} "exact separation":  almost surely, for $N$ large enough, the number of eigenvalues of 
${\boldsymbol \Sigma}_N {\boldsymbol \Sigma}_N^*$ less than $a$ (resp. greater than $b$) coincides with the number of eigenvalues of ${\bf B}_N {\bf B}_N^{*}$ associated to 
the clusters included into $[0, a]$ (resp. included into $[b, \infty)$).
Note that these results also hold in the case where $K=M$, not treated in this paper. Indeed, the analysis of the support $\Scal_N$ provided in \cite{vallet2010sub} can be extended 
when $\B_N\B_N^*$ is full rank. 
Once the characterization of the support is established, the probabilistic part of the proof of the above mentioned exact separation result eigenvalues can be used verbatim.

We also use the separation result to study the case where $\mathrm{Rank}({\bf B}_N) = K$ is independent of $N$. 
It is assumed that for each $k = 1, \ldots, K$, the non zero eigenvalues of ${\bf B}_N {\bf B}_N^{*}$ satisfy $\lim_{N \rightarrow +\infty} \lambda_{k,N} = \lambda_k$.
% \begin{align}
% 	\lim_{N \rightarrow +\infty} \lambda_{k,N} = \lambda_k
% 	\notag
% \end{align}
The support $\Scal_N$ of $\mu_N$ is first characterized in this case, and using the above results related to the almost sure location of the $(\hat{\lambda}_{k,N})_{k=1, \ldots, M}$, 
it is proved that if $\lambda_k > \sigma^{2} \sqrt{c}$, then, 
\begin{align}
	\hat{\lambda}_{k,N} \rightarrow  \frac{(\sigma^2 + \lambda_k)(\sigma^2 c + \lambda_k)}{\lambda_k},
	\label{eq:expre-limit-lambdahat-1}
\end{align}
and that if $\lambda_k \leq \sigma^{2} \sqrt{c}$, then,  
\begin{align}
	\hat{\lambda}_{k,N} \rightarrow \sigma^{2}(1 + \sqrt{c})^{2}.
	\label{eq:expre-limit-lambdahat-2}
\end{align}
This behaviour was first established in \cite{benaych2011singular} using a different approach.  

	\paragraph{Motivations}
	
Our work has been originally motivated by the context of array processing in which the signals transmitted by $K < M$ sources 
are received by an array equiped with $M$ sensors. 
Under certain assumptions, the $M$-dimensional vector ${\bf y}(n)$ received on the sensor array at time $n$ can be written as 
\begin{align}
	{\bf y}(n) = \sum_{k=1}^{K} {\bf d}_k s_k(n) + {\bf v}(n),
	\label{eq:motivation-model}
\end{align}
where each time series $(s_k(n))_{n \in \mathbb{Z}}$ represents a non observable deterministic signal corresponding to source $k$ and where ${\bf d}_k$ is an unknown deterministic 
$M$-dimensional vector depending on the direction of arrival of the $k$-th source. 
$({\bf v}(n))_{n \in \mathbb{Z}}$ is an additive complex white Gaussian noise such that $\mathbb{E}[{\bf v}(n){\bf v}(n)^*] = \sigma^{2} {\bf I}_M$. 
It is clear that \eqref{eq:motivation-model} is equivalent to \eqref{eq:model-infonoise} if we put $\Sigmabs_N = N^{-1/2} \left[ \y(1), \ldots, \y(N) \right]$, 
$\W_N = N^{-1/2} \sigma^{-1} [ \v(1), \ldots, \v(N)]$ and $\B_N = N^{-1/2} \D [\s(1), \ldots, \s(N)]$,
with ${\bf s}(n) = [s_1(n), \ldots, s_K(n)]^{T}$ and ${\bf D}  = [{\bf d}_1, \ldots, {\bf d}_K]$. 

Model \eqref{eq:motivation-model} poses important statistical problems such as detection of the number of sources $K$ or estimation of the direction of arrivals of the $K$ sources. 
A number of estimation schemes based on the eigenvalues and the eigenvectors of matrix $ {\boldsymbol \Sigma}_N {\boldsymbol \Sigma}_N^*$ were developed, 
and analysed if $N \rightarrow +\infty$ while $M$ remains fixed. 
If however $M$ and $N$ are of the same order of magnitude, the above technics may fail, and it is therefore quite relevant to study these
statistical problems in the asymptotic regime $M,N \rightarrow +\infty$ in such a way that $\frac{M}{N} \rightarrow c$, $c \in (0,+\infty)$. 
The number of sources may be constant or scale up with the dimensions $M$ and $N$. 
For this, the first step is to evaluate the behaviour of the eigenvalues of ${\boldsymbol \Sigma}_N {\boldsymbol \Sigma}_N^*$.

\paragraph{About the literature}

{\em Concerning the zero-mean correlated model.} The problems addressed in this paper were studied extensively in the context of the popular zero-mean correlated model defined by 
\begin{align}
	{\boldsymbol \Sigma}_N = {\bf H}_N {\bf W}_N,
	\label{eq:model-centre}
\end{align}  
where ${\bf H}_N$ is a deterministic $M \times M$ matrix and where ${\bf W}_N$ is a random matrix with possibly non Gaussian zero mean variance $\frac{1}{N}$ i.i.d entries. 
The most complete results concerning the almost sure localization of the eigenvalues of ${\boldsymbol \Sigma}_N{\boldsymbol \Sigma}_N^*$ are due to 
Bai-Silverstein \cite{bai1998no,bai1999exact} and were established in the non Gaussian case. 
Spiked models were first proposed by Johnstone \cite{johnstone2001distribution} in the context of \eqref{eq:model-centre} (matrix ${\bf H}_N$ is a diagonal matrix defined as a 
finite rank perturbation of the identity matrix). 
Later, Baik et al. \cite{baikpeche2005} studied, in the complex Gaussian case, the almost sure convergence of the largest eigenvalues of ${\boldsymbol \Sigma}_N{\boldsymbol \Sigma}_N^{*}$ 
and established central limit theorems. 
The analysis of \cite{baikpeche2005} uses extensively the explicit form of the joint probability distribution of the entries of ${\boldsymbol \Sigma}_N$. 
Using the results of \cite{bai1998no,bai1999exact} as well as the characterization of the support of the limiting distribution $\mu_N$ of the empirical eigenvalue distribution 
$\hat{\mu}_N$ (see Silverstein-Choi \cite{silversteinchoi1995}), Baik-Silverstein \cite{baik2006eigenvalues} addressed the non Gaussian case, and showed the almost sure convergence 
of certain eigenvalues of ${\boldsymbol \Sigma}_N{\boldsymbol \Sigma}_N^{*}$.  
Mestre considered in \cite{mestre2008improved} the case where ${\bf H}_N {\bf H}_N^{*}$ has a finite number of different positive eigenvalues having multiplicities converging to 
$+\infty$, and showed how to estimate the eigenvalues of ${\bf H}_N  {\bf H}_N^{*}$ as well as their associated eigenspace. 
Similar ideas were also developed in \cite{mestre2008modified} in order to address the source localization problem in the context of large sensor arrays when the source signals 
are i.i.d. sequences. 
The analysis of Mestre \cite{mestre2008modified,mestre2008improved} is based on the results of \cite{bai1998no,bai1999exact} as well as on the observation that it is possible 
to exhibit contours depending on the Stieljes transform of $\mu_N$, and enclosing each eigenvalue of ${\bf H}_N {\bf H}_N^{*}$. 
Paul studied in \cite{paul2007} the behaviour of the eigenvectors associated to the greatest eigenvalues of a Gaussian spiked model (almost sure convergence and central limit theorems). 
Bai and Yao showed in \cite{bai2008central} that certain eigenvalues of a non Gaussian spiked model satisfy a central limit theorem. 
We finally note that the above results on zero-mean spiked models have been used in the context of source localization (see \cite{abramovich2008,nadler2010non}). 

{\em Concerning the information plus noise model.} Except our paper \cite{vallet2010sub} devoted to the source localization of deterministic sources, 
the almost sure location of the eigenvalues of matrix ${\boldsymbol \Sigma}_N{\boldsymbol \Sigma}_N^*$ was not studied previously. 
In \cite{vallet2010sub}, we however followed partly the work of Capitaine et al. \cite{capitaine2009largest}, devoted to finite rank deformed Gaussian 
(or satisfying a Poincaré inequality) Wigner matrices, which was inspired by previous results of Haagerup and Thorbjornsen \cite{haagerup2005new}.
See also the recent paper \cite{capitaine2010free} in which the rank of the deformation may scale with the size of the matrix.
We used in \cite{vallet2010sub} the same approach to prove that for $N$ large enough, no eigenvalue of ${\boldsymbol \Sigma}_N{\boldsymbol \Sigma}_N^{*}$ is outside the 
support $\Scal_N$ of $\mu_N$. 
In \cite{vallet2010sub}, under the assumption that the eigenvalue $0$ of $\B_N\B_N^*$ is "far enough" from the others, we established a partial result showing that the $M-K$
smallest eigenvalues of $\Sigmabs_N\Sigmabs_N^*$ are almost surely separated from the others.
In the present paper, we prove a general exact separation property extending the result of \cite{baik2006eigenvalues} to the complex Gaussian information plus noise model.

The almost sure behaviour \eqref{eq:expre-limit-lambdahat-1}, \eqref{eq:expre-limit-lambdahat-2}, of the largest eigenvalues of information plus noise spiked models 
appears to be a consequence of the general results of \cite{benaych2011eigenvalues, benaych2011singular} devoted to the analysis of certain random models with additive and/or 
multiplicative finite rank perturbation. 
\eqref{eq:expre-limit-lambdahat-1} and \eqref{eq:expre-limit-lambdahat-2} are therefore not new, but the technics of \cite{benaych2011singular} completely differ from the 
approach used of the present paper which can be seen as an extension to the information plus noise model of the paper \cite{baik2006eigenvalues}.

\paragraph{Organization of the paper}  

In section \ref{sec:support}, we review some results of \cite{dozier2007analysis} and \cite{vallet2010sub} concerning the support $\Scal_N$ of $\mu_N$ 
as well as some useful background material.  As \cite{vallet2010sub} assumed $c_N < 1$, we address the case $c_N = 1$ and prove some extra results concerning the behaviour of the 
Stieltjes transform of $\mu_N$ around $0$. In section \ref{sec:exact}, we prove the analog of the exact separation of \cite{bai1999exact}.  
\cite{capitaine2009largest} generalized the approach of \cite{bai1999exact} to prove this property in the finite rank deformed Wigner model. 
We however show that it is still possible to use again the ideas of \cite{haagerup2005new}. We establish that it is sufficient to prove that the mass (w.r.t. $\mu_N$) of any interval 
$\Ical$ of $\Scal_N$ is equal to the proportion of eigenvalues of  $\B_N \B_N^*$ associated to $\Ical$. 
For this, we evaluate an integral along a certain contour enclosing the eigenvalues of $\B_N \B_N^*$ associated to $\Ical$. 
This contour is the analog of the contour introduced by \cite{mestre2008improved} in the context of model \eqref{eq:model-centre}
and was extensively used in \cite{vallet2010sub}.
Section \ref{section:spike} addresses the behaviour of the largest eigenvalues of an information plus noise spiked model. 
We analyse the support $\Scal_N$ of $\mu_N$, which appears equivalent to evaluate the positive extrema of a certain rational function. 
Using results concerning perturbed third order polynomial equations, it is shown that if $\lambda_k \neq \sigma^2 \sqrt{c}$ for $k=1,\ldots,K$, 
the intervals of $\Scal_N$ are $[\sigma^{2} (1 - \sqrt{c_N})^{2} + \Ocal(1/M), \sigma^{2} (1 + \sqrt{c_N})^{2} + \Ocal(1/M)]$ and
$[\lambda_{k,N}^{-1}(\lambda_{k,N} + \sigma^2 c_N)(\lambda_{k,N} + \sigma^2) - \Ocal^{+}(M^{-1/2}), 
\lambda_{k,N}^{-1}(\lambda_{k,N} + \sigma^2 c_N)(\lambda_{k,N} + \sigma^2) + \Ocal^{+}(M^{-1/2})]$,
where $k$ is any index for which $\lambda_{k,N} > \sigma^{2} \sqrt{c}$, and where $\Ocal^{+}(M^{-1/2})$ represents a positive $\Ocal(M^{-1/2})$ term. 
The results of section \ref{sec:exact} imply immediately \eqref{eq:expre-limit-lambdahat-1} and \eqref{eq:expre-limit-lambdahat-2} when 
$\lambda_k \neq \sigma^2 \sqrt{c}$ for $k=1,\ldots,K$.
If one the $(\lambda_k)_{k=1,\ldots,K}$ is equal to $\sigma^2 \sqrt{c}$, we use an argument similar to Baik-Silverstein \cite{baik2006eigenvalues}, which relies on an eigenvalue 
perturbation technic.

\paragraph{Model and assumptions}

We now summarize the model and assumptions which will be used in the paper, and introduce some definitions. 
Let $M, N, K \in \mathbb{N}^*$ such that $1 \leq K < M$, $K=K(N)$ and $M=M(N)$, functions of $N$ with $c_N = \frac{M}{N} \to c > 0$ as $N \to \infty$.
We consider a $M \times N$ random matrix $\boldsymbol{\Sigma}_N$ defined as
\begin{align}
	\boldsymbol{\Sigma}_N = \B_N + \sigma \W_N,
	\notag
\end{align}
where $\sigma >0$ and $\B_N$ and $\W_N$ satisfy the two following assumptions.
\begin{assumption}
	\label{assumption:norm_spec_BN}
	Matrix $\B_N$ is deterministic and satisfies $\sup_N \|\B_N\| < +\infty$.
\end{assumption}
\begin{assumption}
	\label{assumption:WN_gaussian}
	The entries of matrix $\W_N$ are i.i.d and follow a standard complex normal distribution $\mathcal{CN}(0,\frac{1}{N})$.
\end{assumption}
Note that the Gaussian assumption \textbf{A-2} will be only required in section \ref{sec:exact}.
All the results in section \ref{sec:support} concerning the convergence of the spectral distribution of $\Sigmabs_N\Sigmabs_N^*$ are also valid in the non Gaussian case.
In the following, we study the context where
\begin{assumption}
      \label{assumption:rankB}
	$\B_N \B_N^*$ is rank deficient, and the non zero eigenvalue of $\B_N \B_N^*$ have multiplicity 1. 
\end{assumption}
The assumption on the multiplicities of the eigenvalues of $\B_N\B_N^*$ is not really necessary, but it allows to simplify the notations.  
We denote by $K$ the rank of $\B_N \B_N^*$ ($K$ may depend on $N$), and by $\lambda_{1,N} > \lambda_{2,N} > \ldots > \lambda_{K,N} > 
\lambda_{K+1,N} = \ldots = \lambda_{M,N} = 0$ its eigenvalues.  
We also assume that 
\begin{assumption}
  	$c_N = \frac{M}{N} \leq 1$ for each $N$.
\end{assumption}
This of course implies that $c \leq 1$. Assuming $c_N \leq 1$ does not introduce any restriction because if $c_N > 1$, the eigenvalues of 
$\boldsymbol{\Sigma}_N  \boldsymbol{\Sigma}_N^{*}$ are $0$ with multiplicity $M-N$ as well as the eigenvalues of matrix  $\boldsymbol{\Sigma}_N^{*}  \boldsymbol{\Sigma}_N$. 
The location of this set of eigenvalues can of course be deduced from the results related to $c_N < 1$.

In this paper, $\Ccal_c^{\infty}(\mathbb{R}, \mathbb{R})$ will denote the set of infinitely differentiable functions with compact support, defined from $\mathbb{R}$ to 
$\mathbb{R}$. If $\Acal \subset \mathbb{R}$, $\partial \Acal$ and $\mathrm{Int}(\Acal)$ represent the boundary and the interior of $\Acal$ respectively. 

We finally recall the definition and useful well known properties of the Stieltjes transform, a fundamental tool for the study of the eigenvalues of random matrices.
Let $\mu$ be a positive finite measure on $\mathbb{R}$. We define its Stieltjes transform $\Psi_{\mu}$ as the function
\begin{align}
	\Psi_{\mu}(z) = \int_{\mathbb{R}} \frac{\mathrm{d}\mu(\lambda)}{\lambda - z} \quad \forall z \in \mathbb{C}\backslash\supp(\mu),
	\notag
\end{align}
where $\supp(\mu)$ represents the support of measure $\mu$. We have the following well-known properties
\begin{property}
	\label{property:Stieltjes}
	$\Psi_{\mu}$ satisfies
	\begin{enumerate}
		\item $\Psi_{\mu}$ is holomorphic on $\mathbb{C}\backslash\supp(\mu)$.
		\item \label{property:Im_Stieltjes}
		$z \in \mathbb{C}^+$ implies $\Psi_{\mu}(z) \in \mathbb{C}^+$.
		\item \label{property:Im_Stieltjes_R+} If $\mu(\mathbb{R}^-_*) = 0$, then $z \Psi_{\mu}(z) \in \mathbb{C}^+$ if  $z \in \mathbb{C}^+$ .
	\end{enumerate}
\end{property}

	\section{\texorpdfstring{Characterization of the support $\Scal_N$ of measure $\mu_N$}{Characterization of the support}}	
	\label{sec:support}
	
In this section, we recall some known results of \cite{dozier2007analysis} and \cite{vallet2010sub} related to the support 
$\Scal_N$ of measure $\mu_N$. As we assumed in \cite{vallet2010sub} that $c_N < 1$, we also provide, when it is necessary, some details on the specific 
case $c_N = 1$. 

	\subsection
	{
	        \texorpdfstring
	        {Convergence of the empirical spectral measure $\hat{\mu}_N$ of $\boldsymbol{\Sigma}_N \boldsymbol{\Sigma}_N^*$ torward $\mu_N$}
	        {Convergence of the e.s.d}
	}		
		
We recall that $\hat{\mu}_N$ is defined by  $\hat{\mu}_N = M^{-1} \sum_{i=1}^M \delta_{\hat{\lambda}_{i,N}}$. Its Stieltjes transform $\hat{m}_N$ is given, 
for all $z \in \mathbb{C}\backslash\{\hat{\lambda}_{1,N},\ldots,\hat{\lambda}_{M,N}\}$, by
\begin{align}
	\hat{m}_N(z) = \int_{\mathbb{R}} \frac{\mathrm{d}\hat{\mu}_N (\lambda)}{\lambda - z}.
	\notag
\end{align}
The following result, concerning the convergence of $\hat{m}_N(z)$ can be found in \cite[Th.1.1]{dozier2007empirical}, \cite[Th.7.4]{girko2001canonical} 
(see also \cite[Th.2.5]{hachem2007deterministic} for a more general model).
\begin{theorem}
	\label{theorem:conv_mh_N}	
	It exists a deterministic probability measure $\mu_N$, such that $\hat{\mu}_N - \mu_N \xrightarrow{\mathcal{D}} 0$ as $N \to \infty$ with probability one.
	Equivalently, the Stieltjes transform $m_N$ of $\mu_N$ satisfies $\hat{m}_N(z) - m_N(z) \to 0$ almost surely $\forall z \in \mathbb{C}\backslash\mathbb{R}^+$.	
	Moreover,  $\forall z \in \mathbb{C}\backslash\mathbb{R}^+$, $m_N(z)$ is the unique solution of the equation, 
	\begin{align}
		m_N(z) =
		\frac{1}{M}\mathrm{Tr}
		\left[  
		-z(1+\sigma^{2}c_{N}m_{N}(z))\mathbf{I}_{M}+\sigma^{2}(1-c_{N})\mathbf{I}_{M}+\frac{\mathbf{B}_{N}\mathbf{B}_{N}^{*}}{1+\sigma^{2}c_{N}m_{N}(z)}
		\right]^{-1}
		\label{definition:m_N}
	\end{align}
	satisfying $\Im(z m_N(z)) > 0$ for $z \in \mathbb{C}^+$.
\end{theorem}
The behaviour of the Stieltjes transform $m_N$ around the real axis is fundamental to evaluate the support $\Scal_N$ of $\mu_N$. The following theorem 
is essentially due to \cite{dozier2007analysis}.   
\begin{theorem}
	\label{theorem:properties_m_N}
	\begin{enumerate}
		\item \label{item:continuity_m_N}
		If $c_N < 1$, the limit of $m_N(z)$, as $z \in \mathbb{C}^+$ converges to $x$, exists for each $x \in \mathbb{R}$ and is still denoted by $m_N(x)$.
		If $c_N = 1$, the limit exists for $x \neq 0$. 
		The function $x \rightarrow m_N(x)$ is continuous on $\mathbb{R}$ if $c_N < 1$ and on $\mathbb{R}^{*}$ if $c_N = 1$. 
		It is also continuously differentiable on $\mathbb{R} \backslash \partial \mathcal{S}_N$.  		
		\item \label{item:re(b)}
		If $c_N < 1$, then $\Re(1 + \sigma^2 c_N m_N(z)) \geq 1/2$ for each $z \in \mathbb{C}^{+} \cup \mathbb{R}$, and if $c_N = 1$, this inequality holds on  
		$\mathbb{C}^{+} \cup \mathbb{R}^{*}$. 	
		\item $m_N(x)$ is a solution of \eqref{definition:m_N} for $x \in \mathbb{R}\backslash \partial \mathcal{S}_N$.
		\item Measure $\mu_N$ is absolutely continuous and its density is given by $f_{\mu_N}(x) = \pi^{-1} \Im(m_N(x))$. 
	\end{enumerate}
\end{theorem}
The statements of this theorem are essentially contained in \cite[Th.2.5]{dozier2007analysis} (see also \cite{vallet2010sub} for more details), except item \ref{item:re(b)} 
because it is shown in \cite[Lem.2.1]{dozier2007analysis} that  $\Re(1 + \sigma^2 c_N m_N(z)) \geq 0$. We therefore prove  item \ref{item:re(b)} in the Appendix \ref{sec:proof-re(b)}. 

We note that as $m_N$ is a Stietljes transform, it also satisfies $m_N(z^*) = m_N(z)^*$. Therefore, it holds that $m_N(z) \to m_N(x)^*$ as $z \in \Cbb^- \to x$, for 
$x \in \mathbb{R}$ if $c_N < 1$ and for $x \in \mathbb{R}^{*}$ if $c_N = 1$. 

In the following, we denote by $f_N, \phi_N$ and $w_N$ the functions defined by 
\begin{align}
	f_N(w) &= \frac{1}{M} \Tr \left(\B_N \B_N^* - w \I_M\right)^{-1}, 
	\notag\\
	\phi_{N}(w) &= w \left(1 - \sigma^2 c_N f_N(w) \right)^2 + \sigma^2 (1-c_N) \left(1 - \sigma^2 c_N f_N(w) \right), 
	\label{definition:phi_N} \\
      	w_N(z) &  =  z (1 + \sigma^2 c_N m_N(z))^2 - \sigma^2 (1-c_N) (1 + \sigma^2 c_N m_N(z)). 
      	\notag  
\end{align}
Functions $w_N$ and $\phi_N$ are of crucial importance because, as shown in \cite{vallet2010sub}, the interior of $\Scal_N$ is given by 
$\mathrm{Int}(\Scal_N) = \{ x > 0, \mathrm{Im}(w_N(x)) > 0 \}$ and for each $x \in \mathbb{R} \backslash \partial \Scal_N$, $w_N(x)$ is a solution of the equation $\phi_N(w) = x$. 
The characterization of $\Scal_N$ proposed in \cite{vallet2010sub}, based on a reformulation of the results in \cite[Th.3.2, Th.3.3]{dozier2007analysis}, consists in identifying 
$w_N(x)$ out of the set of solutions of  $\phi_N(w) = x$.

We also note that \eqref{definition:m_N} is equivalent to 
\begin{align}
	\label{eq:canonique-w}
	\frac{m_N(z)}{1+\sigma^{2} c_N m_N(z)} = f_N(w_N(z)),
\end{align}
and that the identity 
\begin{align}
	\label{eq:expre-b-1}
	\frac{1}{1+\sigma^{2} c_N m_N(z)} = 1 - \sigma^{2} c_N f_N(w_N(z))
\end{align}
holds for $z \in \mathbb{C}^{+} \cup \mathbb{R}$ if $c_N<1$, or for $z \in \mathbb{C}^{+} \cup \mathbb{R}^*$ if $c_N=1$.

	\subsection{\texorpdfstring{Properties of $\phi_N$ and $w_N$, and characterization of $\Scal_N$}{Properties of phi and w}}
	\label{section:function_phiN_support}

In this  paragraph, we recall the main properties of functions $\phi_N$ and $w_N$, as well the structure of $\Scal_N$. Lemmas \ref{lemma:property_phi_extrema},  
\ref{lemma:property_w} as well as theorem \ref{theorem:support} are proved in \cite[Prop.3, Th.2]{vallet2010sub} for $c_N < 1$, but the derivations for $c_N=1$ are similar, 
except items \ref{item:left-continuity} and \ref{item:domination_der_w} of lemma \ref{lemma:property_w}. 
\begin{lemma}
\label{lemma:property_phi_extrema}
	\begin{enumerate}
		\item \label{item:nb_positive_extrema} 
		The function $\phi_{N}$ admits $2 Q_N$ non-negative local extrema counting multiplicities (with $1\leq Q_N \leq K+1$) whose preimages are denoted 
		$w_{1,N}^-<0<w_{1,N}^+\leq w_{2,N}^-\ldots\leq w_{Q_N,N}^-<w_{Q_N,N}^+$.
		\item Define $x_{q,N}^-=\phi_{N}(w_{q,N}^-)$ and $x_{q,N}^+=\phi_{N}( w_{q,N}^+)$ for $q=1\ldots Q_N$. Then,
		\begin{align}
			x_{1,N}^- < x_{1,N}^+ \leq x_{2,N}^- < \ldots \leq x_{Q_N,N}^- < x_{Q_N,N}^+,
			\notag
		\end{align}
		and $x_{1,N}^- > 0$ if $c_N < 1$ while $x_{1,N}^- = 0$ if $c_N = 1$. 
		\item \label{item:location_extrema}
		For $q=1,\ldots,Q_N$, each interval $(w_{q,N}^-,w_{q,N}^+)$ contains at least one element of the set of eigenvalues
		$\{\lambda_{1,N},\ldots,\lambda_{K,N}, 0 \}$ and each eigenvalue of $\B_N \B_N^*$ belongs to one of these intervals.
		\item \label{item:croissance-phi} $\phi_{N}$ is increasing on the intervals $(-\infty,w_{1,N}^-]$, 
		$[w_{1,N}^+,w_{2,N}^-],\ldots,[w_{Q_N-1,N}^+,w_{Q_N,N}^-]$, $[w_{Q_N,N}^+,+\infty)$, and moreover
		\begin{align}
			\phi_{N}\left((-\infty,w_{1,N}^-]\right) &= (-\infty,x_{1,N}^-], 
			\notag\\
			\phi_{N}\left(\left[w_{q,N}^+,w_{q+1,N}^-\right]  \right) &= \left[x_{q,N}^+,x_{q+1,N}^-\right]  
			\quad \text{for each} \ q=1,\ldots,Q_N-1, 
			\notag\\
			\phi_{N}\left([w_{Q_N,N}^+,+\infty)\right) &= [x_{Q_N,N}^+,+\infty).
			\notag
		\end{align}
	\end{enumerate}
\end{lemma}
In figure \ref{figure:phi}, we give a typical representation of function $\phi_N$.
\begin{figure}
	\centering
	\includegraphics[scale=1]{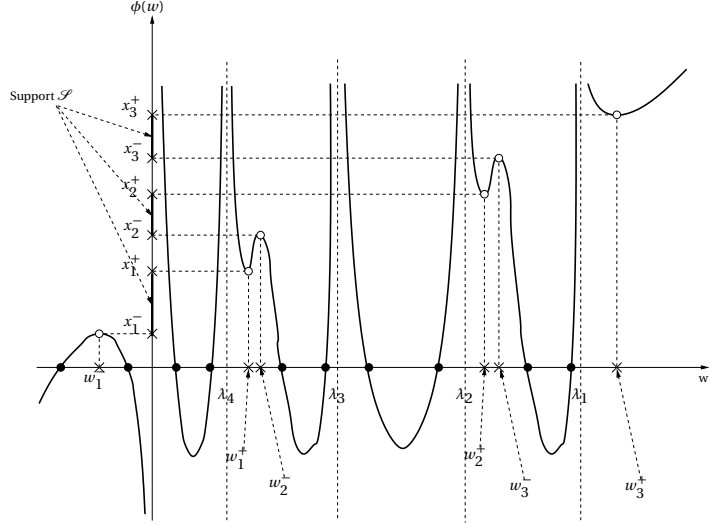}
	\caption{Function $\phi$ for $K=4$ and $c < 1$. Here $Q = 3$.}
	\label{figure:phi}
\end{figure}
We are now in position to recall the characterization of $\Scal_N$ presented in \cite[Th.2]{vallet2010sub}.
\begin{theorem}
	\label{theorem:support}
	The support $\mathcal{S}_N$ is given by
	\begin{align}
		\mathcal{S}_{N}=\bigcup_{q=1}^{Q_N}\left[x_{q,N}^{-},x_{q,N}^{+}\right],
		\notag
	\end{align}
	with $x_{1,N}^{-} = 0$ if $c_N = 1$ and $x_{1,N}^{-} > 0$ if $c_N < 1$.
\end{theorem}
The intervals $([x_{q,N}^{-},x_{q,N}^{+}])_{q=1, \ldots, Q_N}$ will be called the clusters of $\Scal_N$. 
Cluster $[x_{q,N}^{-},x_{q,N}^{+}]$ corresponds to the interval $[w_{q,N}^{-},w_{q,N}^{+}]$ in the sense that $x_{q,N}^{-} = \phi_N(w_{q,N}^{-})$ and $x_{q,N}^{+} = \phi_N(w_{q,N}^{+})$. 
Finally, we shall say that an eigenvalue $\lambda_{k,N}$ of ${\bf B}_N {\bf B}_N^{*}$ is associated to cluster $[x_{q,N}^{-},x_{q,N}^{+}]$ 
if $\lambda_{k,N}  \in (w_{q,N}^{-},w_{q,N}^{+})$. 

In the same way as in theorem \ref{theorem:properties_m_N}, we set $w_N(x) = \lim_{z \in \mathbb{C}^+, z \to x} w_N(z)$ for $x \in \mathbb{R}$ if $c_N < 1$ and for $x \in \mathbb{R}^*$ if
$c_N = 1$. We notice that $\lim_{z \in \mathbb{C}^-, z \to x} w_N(z) = w_N(x)^*$. Function $x \rightarrow w_N(x)$ satisfies the following properties. 
\begin{lemma}
	\label{lemma:property_w}
 	The following properties hold
	\begin{enumerate}
		\item \label{item:regularity_w}
		$x \rightarrow w_N(x)$ is continuous on $\mathbb{R}$ if $c_N < 1$ and on $\mathbb{R}^{*}$ if $c_N = 1$, and is continuously differentiable 
		on $\mathbb{R} \backslash \partial \Scal_N$. 	
		\item \label{item:w_increase}
		$w_N$ is real and increasing on $\mathbb{R}\backslash \mathcal{S}_N$.
		\item \label{item:ineq_f_w}
		$1 - \sigma^2 c_N f_N(w_N(x)) \neq 0$ for $x \in \mathbb{R}\backslash \partial \mathcal{S}_N$.
		\item \label{item:w_on_supp}
		$x \in \mathrm{Int}(\mathcal{S}_N)$ if and only if  $w_N(x) \in \mathbb{C}^+$. 
		\item \label{item:w_eq_phi}
		For $x \in \mathbb{R}\backslash \partial \mathcal{S}_N$, $w_N(x)$ is a solution of the equation $\phi_N(w) = x$. 
		If  $x \in \mathrm{Int}(\mathcal{S}_N)$, $w_N(x)$ is the unique solution belonging to $\mathbb{C}^+$ and if $x \in \mathcal{S}_N^{c}$, 
		$w_N(x)$ is the unique solution satisfying $\phi'_N(w_N(x)) > 0$ and $1-\sigma^2 c_N f_N(w_N(x)) > 0$.
		\item \label{item:left-continuity} 
		Function $x \rightarrow w_N(x)$ is continuous at $x = x_{1,N}^- = 0$ for $c_N = 1$.  
		\item \label{item:w_dS}
		For $q=1,\ldots,Q_N$, $w_N(x_{q,N}^-)=w_{q,N}^-$ and $w_N(x_{q,N}^+)=w_{q,N}^+$.
		\item \label{item:domination_der_w}
		Let $q=1,\ldots,Q_N$. Then, there exists a constant $C > 0$ and  neighborhoods $\mathcal{V}(x_{q,N}^-)$, $\mathcal{V}(x_{q,N}^+)$ of respectively 
		$x_{q,N}^-$ and $x_{q,N}^+$ such that,
		\begin{align}
			|w'_N(x)| & \leq  C \left|x - x_{q,N}^-\right|^{-1/2}
			\quad \forall x \in \mathcal{V}(x_{q,N}^-) \cap \mathbb{R} \backslash \{x_{q,N}^-\}, 
			\label{eq:comportement-w'-moins}
			\\
			|w'_N(x)| & \leq  C \left|x - x_{q,N}^+\right|^{-1/2}
			\quad \forall x \in \mathcal{V}(x_{q,N}^+) \cap \mathbb{R} \backslash \{x_{q,N}^+\},.
			\label{eq:comportement-w'-plus}
		\end{align}			
	\end{enumerate}
\end{lemma}
The lemma was proved in \cite[Prop.2, Lem.3]{vallet2010sub} in the case $c_N < 1$. The proofs extend easily to $c_N = 1$, except items \ref{item:left-continuity} and 
\ref{item:domination_der_w} for $q=1$. These 2 statements are proved in the Appendix \ref{sec:proof-w'}.  

We finish this section by showing that the following result holds. 
\begin{corollary}
        \label{corollary:support_bounded}
	We have
	\begin{align}
		\sup_{N} x_{Q_N,N}^{+} < \infty,
		\notag
	\end{align}
	i.e. $\cup_{N} \Scal_N$ is a bounded set.  
\end{corollary}
\begin{proof} 
	We define $\lambda_{max}$ by $\lambda_{max} = \sup_N \|\B_N\|^2$. It follows that for 
	$w > \lambda_{max}$,
	\begin{align}
		\sup_N |f_N(w)| &\leq \frac{1}{|\lambda_{max} - w|}, 
		\notag\\
		\sup_N |f'_N(w)| &\leq \frac{1}{|\lambda_{max} - w|^2},
		\notag\\
		\sup_N |w f'_N(w)| &\leq \frac{w}{|\lambda_{max} - w|^2},
		\notag
	\end{align}
	and since $\phi'_N(w) = (1-\sigma^2 c_N f_N(w))^2 - 2\sigma^2 c_N w f_N'(w)(1 - \sigma^2 c_N f_N(w)) - \sigma^4 c_N (1-c_N) f'_N(w)$
	converges towards 1 when $w \rightarrow +\infty$, we deduce that for $\epsilon > 0$, $\exists w_{\epsilon} > \lambda_{max}$ such that 
	$\forall w > w_{\epsilon}$, $\phi'_N(w) > \epsilon$ for all $N$.
	Since $\phi'_N(w_{Q_N, N}^+) = 0$, this implies that
	\begin{align}
		\sup_N w_{Q_N,N}^+ \leq w_{\epsilon} < + \infty.
		\notag
	\end{align}
	Moreover, using $w_{Q_N,N}^+ = w_N(x_{Q_N,N}^+) = x_{Q_N,N}^+ (1 + \sigma^2 c_N m_N(x_{Q_N,N}^+))^2 - \sigma^2 (1-c_N)(1 + \sigma^2 c_N m_N(x_{Q_N,N}^+))$,
	and  item \ref{item:re(b)} of theorem \ref{theorem:properties_m_N}, we get that 
	\begin{align}
		x_{Q_N,N}^+ \leq \frac{w_{\epsilon}}{(1 + \sigma^2 c_N m_N(x_{Q_N,N}^+))^{2}} 
		+ \frac{\sigma^{2}(1-c_N)}{1 
		+ \sigma^2 c_N m_N(x_{Q_N,N}^+)} <  4 w_{\epsilon} 
		+ 2 \sigma^{2}.
		\notag
	\end{align}
	This completes the proof.
\end{proof}

\section{Almost sure location of the sample eigenvalues.}
\label{sec:exact}

We first recall the following result of \cite[Th.3]{vallet2010sub}, which states the almost sure absence of eigenvalue of $\Sigmabs_N\Sigmabs_N^*$ outside the support $\Scal_N$ of $\mu_N$ 
for all large $N$. This property is well-known  in the context of zero mean non Gaussian correlated matrices (see \cite{bai1998no}).
We note that the proof of theorem \cite[Th.3]{vallet2010sub} uses extensively that $\W_N$ is Gaussian (assumption \textbf{A-2}).
\begin{theorem}
	\label{theorem:no_eig}
	Let $a,b \in \mathbb{R}$, $\epsilon > 0$ and $N_0 \in \mathbb{N}$ such that $(a - \epsilon, b + \epsilon) \cap \mathcal{S}_N = \emptyset$ for each $N > N_0$.
	Then, with probability one, no eigenvalue of $\boldsymbol{\Sigma}_N \boldsymbol{\Sigma}_N^*$ belongs to $[a,b]$ for $N$ large enough.
\end{theorem}
We remark that theorem \ref{theorem:no_eig} extends to semi-infinite intervals $[b, +\infty)$ because, as $\| {\bf W}_N {\bf W}_N^{*}\| \rightarrow (1+\sqrt{c})^{2}$ almost surely, 
then it holds that $\hat{\lambda}_{1,N} = \|  \boldsymbol{\Sigma}_N \boldsymbol{\Sigma}_N^* \| \leq \sup_N \| \B_N \B_N^{*} \| + 2 \sigma^{2}(1+\sqrt{c})^{2}$ almost surely 
for $N$ large enough. 

In order to interpret this result, assume that for each $N > N_1 \geq N_0$, the number of clusters of $\Scal_N$ does not depend on $N$ (denote $Q$ the number of clusters), 
and that for each $q=1, \ldots, Q$, the sequences $(x_{q,N}^{-})_{N > N_1}$  and $(x_{q,N}^{+})_{N > N_1}$ converge torwards limits $x_{q}^{-}$ and $x_{q}^{+}$ 
satisfying $x_{1}^{-} \leq x_{1}^{+} < x_{2}^{-} \leq  x_{2}^{+} < \ldots <  x_{Q}^{-} \leq x_{Q}^{+}$. 
In this context, theorem \ref{theorem:no_eig} implies that almost surely, for each $\epsilon > 0$, each eigenvalue belongs to one of the intervals $[x_q^{-}-\epsilon, x_{q}^{+}+\epsilon]$ 
for $N$ large enough.

We now establish the following property, also well-know in the literature and referred to as "exact separation" (see e.g. \cite{bai1999exact} in the context of non Gaussian correlated 
zero mean random matrices). 
\begin{theorem}
	\label{theorem:loc_eig}
	Let $a,b \in \mathbb{R}$, $\epsilon > 0$, $N_0 \in \mathbb{N}$ such that $(a- \epsilon, b + \epsilon) \cap \mathcal{S}_N = \emptyset$ for $N > N_0$. 
	Then, with probability one,
	\begin{align}
		\mathrm{card}\{k: \hat{\lambda}_{k,N} < a \} &= \mathrm{card}\{k: \lambda_{k,N} < w_N(a)\} 
		\label{eq:inferieur} \\
		\mathrm{card}\{k: \hat{\lambda}_{k,N} >  b \} &= \mathrm{card}\{k: \lambda_{k,N} > w_N(b)\} 
		%\label{eq:superieur}
		\notag
	\end{align}
	for $N$ large enough.
\end{theorem}
Under the above simplified assumptions, this result means that almost surely for $N$ large enough, the number of sample eigenvalues that belong to each interval 
$[x_q^{-}-\epsilon, x_{q}^{+}+\epsilon]$ coincides with the number of eigenvalues of ${\bf B}_N {\bf B}_N^{*}$ that are associated to the cluster  $[x_{q,N}^{-}, x_{q,N}^{+}]$.
To prove theorem \ref{theorem:loc_eig}, we use the same technic as in \cite{vallet2010sub}, where a less general result is presented in the case $c < 1$.

	\subsection{Preliminary results}
	
We first need to state preliminary useful lemmas. The first lemma is elementary and is related to the solutions of the equation $1 - \sigma^{2} c_N f_N(w) = 0$. 
\begin{lemma}
	\label{lemma:property_phi_zeros}
	The equation $1 - \sigma^{2} c_N f_N(w) = 0$ admits $ K + 1$ real solutions $z_{0,N} < 0 < z_{1,N} <  \lambda_{1,N} < \ldots < z_{K,N} < \lambda_{K,N}$. 
	If $c_N < 1$, $z_{0,N} < w_{1,N}^{-}$ while if $c_N = 1$, $z_{0,N} = w_{1,N}^{-}$. 
	Moreover, for each $k=1,\ldots,K$, each solution $z_{k,N}$ belongs to the interval $(w_{q,N}^-, w_{q,N}^+)$ containing eigenvalue $\lambda_{k,N}$, 
	with $q \in \{1,\ldots,Q_N\}$.
\end{lemma}
The next two lemmas are fundamental, and were proved by Haagerup-Thorbjornsen in \cite{haagerup2005new} in the Wigner case models (see also \cite{capitaine2007freeness}). 
Lemma \ref{lemma:eq_exp_psi} and \ref{lemma:var_psi} are established in \cite[Prop. 4, Lem. 2 and proof of Th.3]{vallet2010sub}.
Note that, unlike section \ref{sec:support}, the Gaussian assumption is required here.
We give here some insights on the proof of these two lemmas for the reader's convenience.
\begin{lemma}
	\label{lemma:eq_exp_psi}
	Let $\psi \in \mathcal{C}_c^\infty(\mathbb{R},\mathbb{R})$, independent of $N$, then
	\begin{align}
		\mathbb{E}\left[\frac{1}{M} \Tr \psi \left(\Sigmabs_N \Sigmabs_N^*\right)\right]
		-
		\int_{\mathcal{S}_N} \psi(\lambda) \mathrm{d}\mu_N(\lambda)
		=
		\mathcal{O}\left(\frac{1}{N^2}\right).
		\notag
	\end{align}
\end{lemma}
\begin{proof}      
        Using the integration by part formula (see e.g. \cite{novikov1965functionals}, \cite{pastur2005simple}) and the Poincaré inequality for Gaussian random vectors 
        \cite{chen1982inequality}, it is proved in \cite[Prop.4]{vallet2010sub} that
        \begin{align}
        	\Ebb\left[\hat{m}_N(z)\right] = m_N(z) + \frac{\chi_N(z)}{N^2},
        	\label{sketch_proof_1}
        \end{align}
        where $\chi_N$ is holomorphic on $\Cbb\backslash\Rbb$ and satisfies
        \begin{align}
	\left|\chi_N(z)\right|\leq \Prm_1\left(|z|\right)\Prm_2\left(\left|\Im(z)^{-1}\right|\right), 
	\label{sketch_proof_3}
        \end{align}
        with $\Prm_1$, $\Prm_2$ two polynomials with positive coefficients independent of $N,z$.
        The Stieltjes inversion formula gives
        \begin{align}
        	\mathbb{E}\left[\frac{1}{M} \Tr \psi \left(\Sigmabs_N \Sigmabs_N^*\right)\right]
        	= 
        	\frac{1}{\pi} \lim_{y\downarrow 0} \Im\left(\int_{\Rbb} \psi(x) \Ebb\left[\hat{m}_N(x+iy)\right]\drm x\right),
        	\label{sketch_proof_2}
        \end{align}
        as well as $\int_{\Rbb} \psi(\lambda)\drm \mu_N(\lambda) = \pi^{-1} \lim_{y\downarrow 0} \Im\left(\int_{\Rbb} \psi(x) m_N(x+iy)\drm x\right)$.
        The polynomial bound \eqref{sketch_proof_3} implies the bound $\limsup_{y \downarrow 0} \int_{\Rbb}\psi(x)\left|\chi_N(x+iy)\right|\drm x \leq C < \infty$, with $C>0$ independent 
        of $N$ (a result shown in \cite[Sec.3.3]{capitaine2007freeness} using the ideas of \cite{haagerup2005new}).
        Plugging \eqref{sketch_proof_1} into \eqref{sketch_proof_2}, we obtain the desired result.
\end{proof}
Lemma \ref{lemma:var_psi} is not explicitely stated  in \cite{vallet2010sub}, but it can be proved easily using the derivation of \cite[eq. (37)]{vallet2010sub}. 
\begin{lemma}
	\label{lemma:var_psi}
	Let $\psi \in \mathcal{C}_c^\infty(\mathbb{R},\mathbb{R})$, independent of $N$ and constant on each cluster of $\mathcal{S}_N$ for $N$ large enough.
	Then, we have
	\begin{align}
		\Var\left[\frac{1}{M} \Tr \psi \left(\Sigmabs_N \Sigmabs_N^*\right)\right] = \mathcal{O}\left(\frac{1}{N^4}\right).
		\notag
	\end{align}
\end{lemma}
\begin{proof}
	We only give a sketch of proof for the reader's convienence. Using the Poincaré inequality for gaussian random vectors, we obtain
	\begin{align}
		\Var\left[\frac{1}{M} \Tr \psi \left(\Sigmabs_N \Sigmabs_N^*\right)\right]
		\leq
		\frac{C}{N^2} \Ebb\left[\frac{1}{M}\Tr \psi'\left(\Sigmabs_N \Sigmabs_N^*\right)^2\Sigmabs_N\Sigmabs_N^*\right]
		=
		\frac{C}{N^2} \left(\int_{\Rbb} \lambda \psi'(\lambda)^2 \drm \mu_N(\lambda) + \Ocal\left(\frac{1}{N^2}\right)\right),
		\notag
	\end{align}
	where the last equality follows from the application of lemma \ref{lemma:eq_exp_psi} to the function $\lambda \mapsto \lambda \psi'(\lambda)^2$. The conclusion follows from the observation 
that this function vanishes on ${\cal S}_N$ for all large $N$. 
\end{proof}

We are now in position to prove theorem \ref{theorem:loc_eig}.

	\subsection{End of the proof}

We first prove \eqref{eq:inferieur} and assume that $a > 0$ because (\ref{eq:inferieur}) is obvious if $a \leq 0$. 
We consider $\eta < \epsilon$ and assume without restriction that $0 < \eta < a$. 
We consider a function $\psi_a \in \mathcal{C}_c^\infty(\mathbb{R},\mathbb{R})$, independent of $N$, such that $\psi_a \in [0,1]$ and
\begin{align}
	\psi_a(\lambda) =
	\begin{cases}
		1 &\quad \forall \lambda \in \left[0, a - \eta\right] 
		\notag \\
		0 &\quad \forall \lambda \geq a.
		\notag
	\end{cases}
\end{align}
By lemma \ref{lemma:eq_exp_psi}, we have
\begin{align}
	\mathbb{E}\left[\frac{1}{M} \Tr \psi_a \left(\boldsymbol{\Sigma}_N \boldsymbol{\Sigma}_N^*\right)\right]
	-
	\int_{\mathbb{R}^+} \psi_a (\lambda) \mathrm{d}\mu_N(\lambda)
	=
	\mathcal{O}\left(\frac{1}{N^2}\right),
	\notag
\end{align}
or equivalently
\begin{align}
	\mathbb{E}\left[\frac{1}{M} \Tr \psi_a \left(\boldsymbol{\Sigma}_N \boldsymbol{\Sigma}_N^*\right)\right]
	=
	\mu_N\left([0,a-\eta]\right) + \mathcal{O}\left(\frac{1}{N^2}\right).
	\notag
\end{align}
Lemma \ref{lemma:var_psi} also implies that
\begin{align}
	\Var\left[\frac{1}{M} \Tr \psi_a \left(\boldsymbol{\Sigma}_N \boldsymbol{\Sigma}_N^*\right)\right] = \mathcal{O}\left(\frac{1}{N^4}\right).
	\notag
\end{align}
Therefore, Markov inequality leads to
\begin{align}
	&\mathbb{P}
	\left(
		\left|\frac{1}{M} \Tr \psi_a \left(\boldsymbol{\Sigma}_N \boldsymbol{\Sigma}_N^*\right) -\mu_N\left([0,a-\eta]\right)\right| > \frac{1}{N^{4/3}}
	\right)
	\notag\\
	&\qquad\leq
	N^{8/3} \Var\left[\frac{1}{M} \Tr \psi_a \left(\boldsymbol{\Sigma}_N \boldsymbol{\Sigma}_N^*\right)\right]
	+
	N^{8/3}\left|\mathbb{E}\left[\frac{1}{M} \Tr \psi_a \left(\boldsymbol{\Sigma}_N \boldsymbol{\Sigma}_N^*\right)-\mu_N\left([0,a-\eta]\right)\right]\right|^2
	\notag\\
	&\qquad= \mathcal{O}\left(\frac{1}{N^{4/3}}\right),
	\notag
\end{align}
which implies that with probability one,
\begin{align}
	\frac{1}{M} \Tr \psi_a \left(\boldsymbol{\Sigma}_N \boldsymbol{\Sigma}_N^*\right) = \mu_N\left([0,a-\eta]\right) + \mathcal{O}\left(\frac{1}{N^{4/3}}\right).
	\label{eq:tr_psi_tmp}
\end{align}
The remainder of the proof is dedicated to the evaluation of $\mu_N([0,a-\eta])$.
Let $I_N = \max\{q: x_{q,N}^+ < a\}$. 
It is clear that $\mu_N([0,a-\eta]) = \sum_{q=1}^{I_N} \mu_N([x_{q,N}^-,x_{q,N}^+])$ because $\mu_N((a-\eta,a)) = 0$. 
By theorem \ref{theorem:properties_m_N}, $\mu_N$ is absolutely continuous with density $\pi^{-1}\Im(m_N(x))$. Therefore, it holds that
\begin{align}
	\label{eq:expre-masse-cluster}
	\mu_N([x_{q,N}^-,x_{q,N}^+]) = \frac{1}{\pi} \Im\left(\int_{x_{q,N}^-}^{x_{q,N}^+} m_N(x) \mathrm{d}x \right).
\end{align}
In order to evaluate the righthandside of \eqref{eq:expre-masse-cluster}, we use the contour integral approach introduced in \cite{vallet2010sub}. 
For this, we consider the curve $\mathcal{C}_{q,N}$ defined by
\begin{align}
	\mathcal{C}_{q,N} = \left\{w_N(x): x \in [x_{q,N}^-,x_{q,N}^+]\right\} \cup \left\{w_N(x)^*: x \in [x_{q,N}^-,x_{q,N}^+]\right\}.
	\notag
\end{align}
We notice that $x \rightarrow w_N(x)$ (resp.  $x \rightarrow w_N(x)^{*}$) is a one-to-one correspondance from 
$(x_{q,N}^{-}, x_{q,N}^{+})$ onto $\{ w_N(x), x \in (x_{q,N}^{-}, x_{q,N}^{+}) \}$ (resp.   $\{ w_N(x)^{*}, x \in (x_{q,N}^{-}, x_{q,N}^{+}) \}$)
because if $w_N(x) = w_N(y)$, then $\phi_N(w_N(x)) = x = \phi_N(w_N(y)) = y$ (see lemma \ref{lemma:property_w}, item \ref{item:w_eq_phi}). 

It follows from lemma \ref{lemma:property_w} items \ref{item:regularity_w},  \ref{item:w_on_supp} and \ref{item:w_dS} that $\mathcal{C}_{q,N}$ is a closed continuous contour enclosing 
the interval $(w_{q,N}^-,w_{q,N}^+)$. $\mathcal{C}_{q,N}$ is differentiable at each point except at $w_{q,N}^{-}$ and $w_{q,N}^{+}$ (see item \ref{item:domination_der_w} of lemma 
\ref{lemma:property_w}). 
However,  \eqref{eq:comportement-w'-moins} and \eqref{eq:comportement-w'-plus} imply that  $|w'_N|$ is summable on $[x_{q,N}^-,x_{q,N}^+]$. 
Therefore, for each function $g$ continuous in a neighborhood of $\mathcal{C}_{q,N}$, satisfying $(g(w))^{*} = g(w^{*})$,  it is still possible to define the contour integral 
$\oint_{\mathcal{C}_{q,N}^-} g(w) dw$ by
\begin{align}
	\oint_{\mathcal{C}_{q,N}^-} g(w) dw  =  2i \mathrm{Im} \left( \int_{x_{q,N}^-}^{x_{q,N}^+} g(w_N(x)) w'_N(x) \mathrm{d}x \right).
	\notag
\end{align}
The notation $\mathcal{C}_{q,N}^-$ means that the contour $ \mathcal{C}_{q,N}$ is clockwise oriented. 
Although $\mathcal{C}_{q,N}$ is not differentiable, the main results related to contour integrals of meromorphic functions remain valid. 
In particular, it holds that
\begin{align}
	\mathrm{Ind}_{\mathcal{C}_{q,N}^-} (\xi) = \frac{1}{2 \pi i} \int_{\mathcal{C}_{q,N}^-} \frac{\mathrm{d}\lambda}{\xi - \lambda}
	=
	\begin{cases}
		1 &\quad \text{if} \ \xi \in \left(w_{q,N}^-,w_{q,N}^+ \right)
		\notag \\
		0 & \quad \text{if} \ \xi \not \in \left[w_{q,N}^-,w_{q,N}^+ \right]
		\notag
	\end{cases}
	\notag
\end{align}
In order to evaluate the righthandside of \eqref{eq:expre-masse-cluster} using a contour integral, we remark that
\begin{align}
	m_N(x) = \frac{f_N(w_N(x))}{1 - \sigma_N^2 c_N f_N(w_N(x))} \quad \forall x \in \mathbb{R}\backslash \partial\mathcal{S}_N
	\notag
\end{align}
(see \eqref{eq:canonique-w} and item \ref{item:ineq_f_w} of lemma \ref{lemma:property_w}). 
Moreover, by item \ref{item:w_eq_phi} of lemma \ref{lemma:property_w}, we have $w_N'(x) \phi'_N(w_N(x)) = 1$ on $(x_{q,N}^-,x_{q,N}^+)$. 
Therefore, we have
\begin{align}
	\mu_N([x_{q,N}^-,x_{q,N}^+]) = 
	\frac{1}{\pi} 
	\Im
	\left(
		\int_{x_{q,N}^-}^{x_{q,N}^+}  g_N(w_N(x)) w'_N(x) \mathrm{d}x 
	\right),
	\label{eq:expre-masse-integrale-1}
\end{align}
where $g_N(w)$ is the rational function defined by
\begin{align}
	g_N(w) =  
	\frac{f_N(w) \phi'_N(w)}{1 - \sigma_N^2 c_N f_N(w)} 
	= 	
	f_N(w) 
	\frac{(1 - \sigma^2 c_N f_N(w))^2 - 2 \sigma_N^2 c_N w f'_N(w) (1 - \sigma_N^2 c_N f_N(w)) - \sigma_N^4 c_N (1-c_N) f_N'(w)}{1 - \sigma_N^2 c_N f_N(w)}.
	\notag
\end{align}
In order to justify the existence of the integral at the righthandside of \eqref{eq:expre-masse-integrale-1}, we prove that $g_N(w)$ is continuous in a neighborhood of $\Ccal_{q,N}$. 
We first note that the poles of $g_N(w)$ coincide with the eigenvalues of ${\bf B}_N {\bf B}_N^{*}$ and the zeros $(z_{k,N})_{k=0, \ldots, K}$ of $1 - \sigma_N^2 c_N f_N(w)$. 
As $w_N(x)$ is not real on  $(x_{q,N}^-,x_{q,N}^+)$, $x \rightarrow g_N(w_N(x))$ is continuous on  $(x_{q,N}^-,x_{q,N}^+)$. 
It remains to check the continuity at $x_{q,N}^-$ and $x_{q,N}^+$. 
If $c_N < 1$, $w_{q,N}^{-} = w_N(x_{q,N}^-)$ and $w_{q,N}^{+} = w_N(x_{q,N}^+)$ do not coincide with one the poles of $g_N(w)$. 
If $c_N = 1$ and $q=1$, this property still holds true except for $w_{1,N}^{-} = w_N(x_{1,N}^-) = w_N(0)$ because $z_{0,N} = w_{1,N}^{-}$ (see lemma \ref{lemma:property_phi_zeros}). 
However, if $c_N = 1$, the solutions of $1 - \sigma_N^2 c_N f_N(w)$ are not poles of $g_N$ due to a pole zero cancellation.

Therefore, it is clear that $\mu_N([x_{q,N}^-,x_{q,N}^+])$ can also be written as
\begin{align}
	\mu_N([x_{q,N}^-,x_{q,N}^+]) = 
	\frac{1}{2\pi i} \oint_{\mathcal{C}_{q,N}^-} g_N(\lambda) \mathrm{d}\lambda.
	\notag
\end{align}
The integral can be evaluated using residue theorem and we give here the main steps of calculation.
Define $\mathcal{I}_q = \{k \in \{1,2, \ldots, K\}: \lambda_{k,N} \in (w_{q,N}^-,w_{q,N}^+)\}$ and $L_q=\mathrm{card}(\mathcal{I}_q) > 0$ 
($L_q > 0$ from lemma \ref{lemma:property_phi_extrema} item \ref{item:location_extrema}).
Assume $c_N < 1$. Since $\mathcal{C}_{q,N}$ only encloses $(w_{q,N}^-,w_{q,N}^+)$, we will have residues at the following points:
\begin{itemize}
	\item for $q=1$: residues at $z_{0,N}$, $0$ and $z_{k,N},\lambda_{k,N}$ for $k \in \mathcal{I}_1$.
	\item for $q \geq 2$: residues at $z_{k,N},\lambda_{k,N}$ for $k \in \mathcal{I}_q$.
\end{itemize}
If $c_N=1$, the zeros of $1 - \sigma^{2} c_N f_N(w)$ are not poles of $g_N(w)$:
\begin{itemize}
	\item for $q=1$: residues at $0$ and $\lambda_{k,N}$ for $k \in \mathcal{I}_1$.
	\item for $q \geq 2$: residues at $\lambda_{k,N}$ for $k \in \mathcal{I}_q$.
\end{itemize}
We just consider the case $c_N < 1$ in the following (the calculations are  similar for $c_N=1$ and are therefore omitted).
We consider the decomposition $g_N(\lambda) = g_{1,N}(\lambda) + g_{2,N}(\lambda) + g_{3,N}(\lambda)$, with
\begin{align}
	g_{1,N}(\lambda) &= f_N(\lambda) \left(1 - \sigma^2 c_N f_N(\lambda) \right), 
	\notag\\
	g_{2,N}(\lambda) &= -2 \sigma^2 c_N \lambda f_N(\lambda) f'_N(\lambda),  
	\notag\\
	g_{3,N}(\lambda) &= -\sigma^4 c_N (1 - c_N) \frac{ f_N(\lambda) f'_N(\lambda)}{1 - \sigma^2 c_N f_N(\lambda)}.
	\notag
\end{align}
These three functions admit poles at $0$,$\left(\lambda_{k,N}\right)_{k=1,\ldots,K}$, and $g_{3,N}$ has moreover poles at $(z_{k,N})_{k=0, \ldots, K}$.
After tedious but straightforward calculations, we finally find that for $k \in \{ 1,2, \ldots, K \}$, 
\begin{align}
	\mathrm{Res}\left(g_{1,N},\lambda_{k,N}\right) &= - \frac{1}{M} + \frac{2 \sigma^2 c_N}{M^2} \sum_{l \neq k} \frac{1}{\lambda_{l,N} - \lambda_{k,N}},
	\notag\\
	\mathrm{Res}\left(g_{2,N},\lambda_{k,N}\right) &= - \frac{2 \sigma^2 c_N}{M^2} \sum_{l \neq k} \frac{1}{\lambda_{l,N} - \lambda_{k,N}},
	\notag\\
	\mathrm{Res}\left(g_{3,N},\lambda_{k,N}\right) &= - \frac{1 - c_N}{c_N}.
	\notag
\end{align}
For the residues at $0$, we get
\begin{align}
	\mathrm{Res}\left(g_{1,N},0\right) &= -\frac{M-K}{M} + 2 \sigma^2 c_N \frac{M-K}{M} \frac{1}{M}\sum_{l=1}^K \frac{1}{\lambda_{l,N}},
	\notag\\
	\mathrm{Res}\left(g_{2,N},0\right) &= -2 \sigma^2 c_N \frac{M-K}{M} \frac{1}{M} \sum_{l=1}^K \frac{1}{\lambda_{l,N}},
	\notag \\
	\mathrm{Res}\left(g_{3,N},0\right) &= -\frac{1 - c_N}{c_N}.
	\notag
\end{align}
Finally, the residues at $z_{k,N}$ for $k=0, \ldots, K$ are given by $\mathrm{Res}(g_{3,N},z_{k,N}) = \frac{1 - c_N}{c_N}$.
Using these evaluations, we obtain immediately that if $q \geq 2$, then, 
\begin{align}
	\mu_N([x_{q,N}^-,x_{q,N}^+]) &= 
	- \sum_{k \in \mathcal{I}_q} 
	\left[
		\mathrm{Res}\left(g_{1,N},\lambda_{k,N}\right)
		+
		\mathrm{Res}\left(g_{2,N},\lambda_{k,N}\right)
		+
		\mathrm{Res}\left(g_{3,N},\lambda_{k,N}\right)
		+
		\mathrm{Res}\left(g_{3,N},z_{k,N}\right)
	\right]
	= \frac{L_q}{M}.
	\notag
\end{align}
This coincides with the ratio of eigenvalues of ${\bf B}_N {\bf B}_N^{*}$ associated to the cluster $[x_{q,N}^-,x_{q,N}^+]$ (i.e. the eigenvalues 
$\lambda_{k,N}$ for $k\in\mathcal{I}_q$). 
If $q = 1$, 
\begin{align}
	\mu_N([x_{1,N}^-,x_{1,N}^+]) &= 
	- \sum_{k \in \mathcal{I}_1} 
	\left[
		\mathrm{Res}\left(g_{1,N},\lambda_{k,N}\right)
		+
		\mathrm{Res}\left(g_{2,N},\lambda_{k,N}\right)
		+
		\mathrm{Res}\left(g_{3,N},\lambda_{k,N}\right)
		+
		\mathrm{Res}\left(g_{3,N},z_{k,N}\right)
	\right]
	\notag\\
	& \qquad -\left[ \mathrm{Res}\left(g_{1,N},0\right) + \mathrm{Res}\left(g_{2,N},0\right)  + \mathrm{Res}\left(g_{3,N},0\right) + 
	\mathrm{Res}\left(g_{3,N},z_{0,N}, \right) \right] 
	\notag\\
	& = \frac{L_1}{M} + \frac{M-K}{M},
	\notag
\end{align}
which also coincides with the ratio of eigenvalues of ${\bf B}_N {\bf B}_N^{*}$ associated to the cluster $[x_{1,N}^-,x_{1,N}^+]$
(the $\lambda_{k,N}$ for $k\in\mathcal{I}_1$ and $0$ with multiplicity $M-K$). 

Therefore, using \eqref{eq:tr_psi_tmp}, we get that
\begin{align}
	\label{eq:entier}
	\Tr \psi_a \left(\boldsymbol{\Sigma}_N \boldsymbol{\Sigma}_N^*\right) - \left( \sum_{q=1}^{I_N} L_q  +(M-K) \right) = \mathcal{O}\left(\frac{1}{N^{1/3}}\right).
\end{align}
But almost surely, for $N$ large enough, $\Tr \psi_a (\boldsymbol{\Sigma}_N \boldsymbol{\Sigma}_N^*)$ is exactly the number of eigenvalues contained in $[0,a]$ 
because no eigenvalue of $\boldsymbol{\Sigma}_N \boldsymbol{\Sigma}_N^*$ belong to $[a-\eta,a]$ (use theorem \ref{theorem:no_eig} with $a-\eta$ in place of $a$). 
The left handside of \eqref{eq:entier} is thus an integer. 
Since this integer decreases at rate $N^{-1/3}$, it is equal to zero for $N$ large enough.
\eqref{eq:inferieur} follows from the observation that $\sum_{q=1}^{I_N} L_q + M-K$ is equal to the number of eigenvalues of ${\bf B}_N {\bf B}_N^{*}$ that are less than $w_N(a)$. 

To evaluate the number of eigenvalues in interval $(b, +\infty)$, we use that no eigenvalue belongs to $[a,b]$ (theorem \ref{theorem:no_eig}). 
Therefore, 
\begin{align}
	\mathrm{card} \{ k: \hat{\lambda}_{k,N} > b \} = M - \mathrm{card} \{ k: \hat{\lambda}_{k,N} < a \}.
	\notag
\end{align}
\eqref{eq:inferieur} implies that 
\begin{align}
	\mathrm{card} \{ k: \hat{\lambda}_{k,N} > b \} = M-\sum_{q=1}^{I_N} L_q - (M-K),
	\notag
\end{align}
which coincides with the number of eigenvalues of $\B_N \B_N^*$ in interval $(w_N(b),+\infty)$. This concludes the proof of theorem \ref{theorem:loc_eig}.

	\section{Applications to the spiked models}
	\label{section:spike}
	
In this section, we use the above results in order to evaluate the behaviour of the largest eigenvalues of 
the information plus noise spiked models. In the remainder of this section, we assume that
\begin{assumption}
	\label{assumption:spiked_model}
	$K$ does not depend on $N$ and for all $k=1,\ldots,K$, the positive sequence $(\lambda_{k,N})$ writes
	\begin{align}
		\lambda_{k,N} = \lambda_k + \varepsilon_{k,N},
		\notag
	\end{align}
	with $\lim_{N \rightarrow +\infty} \varepsilon_{k,N} = 0$ and $\lambda_i \neq \lambda_j$ for $i \neq j$.
\end{assumption}
We define $K_s = \max\{k : \lambda_k > \sigma^2 \sqrt{c}\}$ and the function $\psi(\lambda) = \frac{(\sigma^2 + \lambda)(\sigma^2 c + \lambda)}{\lambda}$.
In the following, we characterize the support $\Scal_N$ of measure $\mu_N$ and use the above results on the almost sure location of the sample eigenvalues in order to prove the theorem
\begin{theorem}
	\label{theorem:lim_spiked_eig}
	We have with probability one,
	\begin{align}
		\hat{\lambda}_{k,N} \xrightarrow[N \to \infty]{}
		\begin{cases}
			\psi(\lambda_k) & \quad \text{if} \quad k \leq K_s \\
			\sigma^2 (1+\sqrt{c})^2  & \quad k \in \{K_s + 1,\ldots,K\}	
		\end{cases}
		\notag
	\end{align}
\end{theorem}
We note that theorem  \ref{theorem:lim_spiked_eig} was already proved in the recent paper \cite{benaych2011singular} using a different approach.

	\subsection{Preliminary results on perturbed equations}

We first state two useful lemmas related to the solutions of perturbed equations. 
They can be interpreted as extensions of lemmas 3.2 and 3.3 of  \cite{baik2006eigenvalues}. 
In the following, we denote respectively by $\mathcal{D}_{o}(z,r)$, $\mathcal{D}_{c}(z,r)$ and $\mathcal{C}(z,r)$ the open disk, closed disk and circle of radius $r>0$ with center $z$. 
Moreover, in this paragraph, the notation $o(1)$ denotes a term that converges towards 0 when the variable $\epsilon$ converges towards $0$. 
The first result is a straightforward modification of \cite[lemma 3.2]{baik2006eigenvalues}. Its proof is thus omitted. 	
\begin{lemma}
        \label{lemma:eq_pert_1}
        For each $\epsilon > 0$, we consider $h_{\epsilon}(z)= h(z) + \chi_{\epsilon}(z)$ with $h, \chi_{\epsilon}$ two holomorphic functions in a disk 
        $\mathcal{D}_{o}(z_0,r_0)$. 
        We assume that $\sup_{z \in \mathcal{D}_{o}(z_0,r_0)} |\chi_{\epsilon}(z)| = o(1)$. We consider  $z_{0,\epsilon} = z_0 + \delta_{\epsilon}$ 
        with $\delta_{\epsilon}=o(1)$. 
        Then, $\exists \ \epsilon_0 > 0$ and $r > 0$ such that for each $0  < \epsilon \leq \epsilon_0$, 
        $z_{0,\epsilon} \in \mathcal{D}_{o}(z_0,r)$ and 
        the equation
        \begin{align}
	       z - z_{0,\epsilon} - \epsilon h_{\epsilon}(z) = 0,
	       \notag
        \end{align}
        admits a unique solution in $\mathcal{D}_{o}(z_{0},r)$ given by
        \begin{align}
	        z_{\epsilon} &= z_{0,\epsilon} + \epsilon h(z_{0}) + o(\epsilon).
	        \notag
        \end{align}
        Moreover, if we assume that $z_0 \in \mathbb{R}$, $h(z) \in \mathbb{R}$ for $z \in \mathbb{R}$, and that for $\epsilon$ small enough,
        $z_{0,\epsilon} \in \mathbb{R}$, $h_{\epsilon}(z) \in \mathbb{R}$ for $z \in \mathbb{R}$, then $z_{\epsilon} \in \mathbb{R}$.
\end{lemma}
The second result is an extension of \cite[Lem.3.3]{baik2006eigenvalues} to certain third degree equations. The proof is given the Appendix \ref{section:proof_lemma_eq_pert_3}.
\begin{lemma}
	\label{lemma:eq_pert_3}
	For each $\epsilon > 0$ and $i=1,2$, we consider $h_{i,\epsilon}(z)= h_{i}(z) + \chi_{i,\epsilon}(z)$ with $h_{i}, \chi_{i,\epsilon}$ holomorphic functions 
	in a disk  $\mathcal{D}_{o}(z_0,r_0)$. 
	We assume that  $h_1(z_0) \neq 0$ and that $\sup_{z \in \mathcal{D}_{o}(z_0,r_0)} |\chi_{i,\epsilon}(z)| = o(1)$ for $i=1,2$.   
	We consider $z_{0,\epsilon} = z_0 + \delta_{\epsilon}$ with $\delta_{\epsilon}=o(1)$.
	Then, $\exists \ \epsilon_0 > 0$ and $r > 0$ such that $z_{0,\epsilon} \subset \mathcal{D}_{o}(z_0,r)$ $\forall \epsilon \in (0, \epsilon_0)$ and the equation
	\begin{align}
		\left(z - z_{0,\epsilon}\right)^3 - \epsilon \left(z - z_{0,\epsilon}\right) h_{1,\epsilon}(z) +\epsilon^2 h_{2,\epsilon}(z) = 0
		\notag
	\end{align}
	has $3$ solutions in  $\mathcal{D}_{o}(z_{0},r)$ given by
	\begin{align}
		z_{\epsilon}^- 
		&= 
		z_{0,\epsilon} - \sqrt{\epsilon} \sqrt{h_{1}(z_{0})} 
		+ o(\sqrt{\epsilon})
		\notag\\
		z_{\epsilon}^+ 
		&= 
		z_{0,\epsilon} 
		+ \sqrt{\epsilon} \sqrt{h_{1}(z_{0})} 
		+  o(\sqrt{\epsilon})
		\notag\\
		z_{\epsilon} &= 
		z_{0,\epsilon} + \epsilon \frac{h_{2}(z_{0})}{h_{1}(z_{0})} 
		+ o(\epsilon),
		\notag
	\end{align}
	where $\sqrt{.}$ is an arbitrary branch of the square root, analytic in a neighborhood of $h_{1}(z_{0})$.
	Moreover, if we assume that $z_{0} \in \mathbb{R}$, $h_i(z) \in \mathbb{R}$ for $z \in \mathbb{R}$ and that for $\epsilon$ small enough that
	$z_{0,\epsilon} \in \mathbb{R}$, $h_{i,\epsilon}(z) \in \mathbb{R}$ for $z \in \mathbb{R}$, then $z_{\epsilon}$ is real. 
	Moreover, if $h_1(z_{0}) > 0$  then $z_{\epsilon}^-$, $z_{\epsilon}^+$ and $z_{\epsilon}$ are real while $z_{\epsilon}^-$, $z_{\epsilon}^+$ are non real 
	if $h_1(z_{0}) < 0$. 
\end{lemma}

\subsection{\texorpdfstring{Characterization of $\Scal_N$ and limits of the eigenvalues if $\lambda_k \neq \sigma^2 \sqrt{c}$}{Characterization of the support and limits of eigenvalues}}
	
	\label{sec:proof_spike1}
	
In this paragraph, we identify the clusters of the support $\Scal_N$, and evaluate the points $x_{q,N}^-, x_{q,N}^+$
for $q=1,\ldots,Q_N$. From theorem \ref{theorem:support}, these points coincide with the positive extrema of function $\phi_N$ (defined in \eqref{definition:phi_N}), and
eventually $x_{1,N}^- = 0$ if $c_N = 1$. Therefore, we first evaluate the real zeros of 
$\phi'_N (w) = (1 - \sigma^2 c_N f_N(w))^2 - 2 \sigma^2 c_N w f'_N(w)(1 - \sigma^2 c_N f_N(w)) - \sigma^4 c_N (1-c_N) f'_N(w)$.
Straightforward calculations give 
\begin{align}
	\phi'_N(w) = \frac{1}{w^2 \prod_{k=1}^K (\lambda_{k,N} - w)^3}\left[ \gamma_{1,N}(w) + \frac{1}{M} \gamma_{2,N}(w) + \frac{1}{M^2} \gamma_{3,N}(w) \right],
	\notag
\end{align}
with
\begin{align}
	\gamma_{1,N}(w) &=
	(w^2 - \sigma^4 c_N) \prod_{k=1}^K (\lambda_{k,N}-w)^3,
	\notag\\
	\gamma_{2,N}(w) &=
	- 2 \sigma^2 c_N 
	\prod_{k=1}^K 
	(\lambda_{k,N}-w)
	\sum_{j=1}^K 
	\left[ 
		\lambda_{j,N}
		\left(w^2 + \sigma^2 (1+c_N)w - \frac{\sigma^2 (1+c_N)\lambda_{j,N}}{2} \right)
		\prod_{\substack{l=1 \\ l \neq j}}^K (\lambda_{l,N}-w)^2
	\right],
	\notag\\
	\gamma_{3,N}(w) &=
	\sigma^4 c_N^2
	\left(
		\sum_{k=1}^K \lambda_{k,N} \prod_{\substack{l=1 \\ l \neq k}}^K (\lambda_{l,N} - w)\right)
	\left(	\sum_{k=1}^K \lambda_{k,N} (3 w - \lambda_{k,N}) \prod_{\substack{l=1 \\ l \neq k}}^K (\lambda_{l,N} - w)^2
	\right).
	\notag
\end{align}
Therefore, $\phi_N^{'}(w) = 0$ if and only if
\begin{align}
	\gamma_{1,N}(w) + \frac{1}{M} \gamma_{2,N}(w) + \frac{1}{M^2} \gamma_{3,N}(w) = 0.
	\label{equation:zeros_gamma}
\end{align}
We assume $c < 1$, which implies that $c_N < 1$ for $N$ large enough.  The calculations are essentially the same if $c=1$.
We first observe that the zeros of $\phi_N$ are included into a compact interval $\Ical$ independent of $N$ (see the proof of 
Corollary \ref{corollary:support_bounded}). Next, we claim that for each $\alpha > 0$, it exists $\beta > 0$ and $N_0 \in \mathbb{N}$ such that 
\begin{align}
	\left| \gamma_{1,N}(w) + \frac{1}{M} \gamma_{2,N}(w) + \frac{1}{M^2} \gamma_{3,N}(w) \right|  > \beta,
	\notag
\end{align}
if $N > N_0$ and $|w-\sigma^{2} \sqrt{c}| > \alpha, |w+\sigma^{2} \sqrt{c}| > \alpha, |w - \lambda_k| > \alpha, k=1, \ldots, K$ and $w \in \Ical$.   
This follows immediately from the inequality 
\begin{align}
	\left| \gamma_{1,N}(w) + \frac{1}{M} \gamma_{2,N}(w) + \frac{1}{M^2} \gamma_{3,N}(w) \right| \geq 
	|\gamma_{1,N}(w)| -  \frac{1}{M} \gamma_{2,max} - \frac{1}{M^{2}}  \gamma_{3,max},
	\notag
\end{align}
where $\gamma_{i,max} = \max_{w \in \Ical} | \gamma_{i,N}(w)|$ for $i=2,3$.  This shows that the solutions of eq. \eqref{equation:zeros_gamma}
are located around the points $\sigma^{2} \sqrt{c}, -\sigma^{2} \sqrt{c}, \lambda_k, k=1, \ldots, K$.

In a disk $\Dcal_o(\sigma^{2} \sqrt{c},r)$, \eqref{equation:zeros_gamma} is equivalent to 
\begin{align}
	w - \sigma^{2} \sqrt{c}_N + \frac{1}{M} \frac{w - \sigma^{2} \sqrt{c}_N}{\gamma_{1,N}(w)} \left(\gamma_{2,N}(w) + \frac{1}{M} \gamma_{3,N}(w)\right) = 0.
	\label{eq:+sigma2racinec}
\end{align}
We use lemma \ref{lemma:eq_pert_1} with $\epsilon = M^{-1}$, $z_0 = \sigma^{2} \sqrt{c}$, $z_{0,\epsilon} =  \sigma^{2} \sqrt{c}_N$, and 
the functions 
\begin{align}
        h_{\epsilon}(w) =  -  \frac{(w - \sigma^2 \sqrt{c_N})}{\gamma_{1,N}(w)} \left[\gamma_{2,N}(w) + \frac{1}{M} \gamma_{3,N}(w)\right]
        \notag
\end{align}
and $h(w) = \lim_{M \rightarrow +\infty} h_{\epsilon}(w)$. $h(w)$ is  obtained by replacing $c_N$ and the $(\lambda_{k,N})_{k=1, \ldots, K}$  by $c$ and the 
$(\lambda_{k})_{k=1, \ldots, K}$ in the expression of $h_{\epsilon}$.  
Lemma \ref{lemma:eq_pert_1} implies that it exists $r$ for which equation \eqref{eq:+sigma2racinec}, or equivalently  equation \eqref{equation:zeros_gamma}, has a unique solution 
in $\Dcal_o(\sigma^{2} \sqrt{c},r)$ for $M$ large enough. 
This solution is given by $\sigma^2 \sqrt{c_N} + \mathcal{O}(\frac{1}{M})$.
It is easy to check that 
\begin{align}
	\phi_N \left(\sigma^2 \sqrt{c_N} + \mathcal{O}\left(M^{-1}\right) \right)  = \sigma^{2}(1 + \sqrt{c}_N)^{2} + \mathcal{O}\left(\frac{1}{M}\right).
	\notag
\end{align}
This quantity is positive, thus showing that $\sigma^2 \sqrt{c_N} + \mathcal{O}(M^{-1})$ is the pre-image  of a positive extremum of $\phi_N$. 
Exchanging $\sigma^{2} \sqrt{c}$ with $-\sigma^{2} \sqrt{c}$, we obtain similarly that it exists a neighborhood of  $-\sigma^{2} \sqrt{c}$ in which 
equation \eqref{equation:zeros_gamma} has a unique solution given by $-\sigma^2 \sqrt{c_N} + \mathcal{O}(\frac{1}{M})$.
Moreover,
\begin{align}
	\phi_N \left( -\sigma^2 \sqrt{c_N} + \mathcal{O}\left(M^{-1}\right) \right) = \sigma^{2}(1 - \sqrt{c}_N)^{2} + \mathcal{O}\left(\frac{1}{M}\right),
	\notag
\end{align}
so that $-\sigma^2 \sqrt{c_N} + \mathcal{O}(\frac{1}{M})$ is also the pre-image of a positive extremum of $\phi_N$.

We now consider $i \in \{1, \ldots, K \}$, and study the equation  \eqref{equation:zeros_gamma} in a neighborhood $\Dcal_o(\lambda_i,r)$ of $\lambda_i$. 
In order to use lemma \ref{lemma:eq_pert_3}, we put $\epsilon = \frac{1}{M}, z_0 = \lambda_i, z_{0,\epsilon} = \lambda_{i,N}$. It is easily seen 
that in  $\Dcal_o(\lambda_i,r)$, eq. \eqref{equation:zeros_gamma} is equivalent to 
\begin{align}
	(w - \lambda_{i,N})^{3} - \frac{1}{M} (w - \lambda_{i,N}) h_{1,\epsilon}(w) +  \frac{1}{M^{2}} h_{2,\epsilon}(w) = 0,
	\notag
\end{align}
where 
\begin{align}
	h_{1,\epsilon}(w)
	&=
	\frac
	{
		2 \sigma^2 c_N \sum_{k=1}^N 
		\left[ 
			\lambda_{k,N}
			\left(w^2 + \sigma^2 (1+c_N)w - \frac{\sigma^2 (1+c_N)\lambda_{k,N}}{2} \right)
			\prod_{\substack{l=1 \\ l \neq k}}^K (\lambda_{l,N}-w)^2
		\right]
	}
	{
		(w^2 - \sigma^4 c_N) \prod_{\substack{k=1 \\ k \neq i}}^K (\lambda_{k,N} - w)^2
	},
	\notag\\
	h_{2,\epsilon}(w) &= - \frac{\gamma_{3,N}(w)}{(w^{2} - \sigma^{4} c_N) \prod_{\substack{k \neq i}}^K (\lambda_{k,N}-w)^3}.
	\notag
\end{align}
We denote by $h_1(w)$ and $h_2(w)$ the limits of $h_{1,\epsilon}(w)$ and $h_{2,\epsilon}(w)$ when $\epsilon \rightarrow 0$, i.e. 
the functions obtained by replacing $c_N$ and the $(\lambda_{k,N})_{k=1, \ldots, K}$  by $c$ and the $(\lambda_{k})_{k=1, \ldots, K}$ respectively  in the expressions of 
$h_{1,\epsilon}, h_{2,\epsilon}$. After some algebra, we obtain that
\begin{align}	
	h_{1}(\lambda_{i}) = \frac{2 \sigma^{2} c \lambda_i^{2}(\lambda_i + \frac{\sigma^{2}(1+c)}{2})}{\lambda_i^{2} - \sigma^{4} c},
	\notag
\end{align}
while $h_{2}(\lambda_i)$ is equal to 
\begin{align}	
	h_{2}(\lambda_i) =  -\frac{2 \sigma^4 c^2 \lambda_i^{3}}{\lambda_i^{2} - \sigma^{4} c}.
	\notag
\end{align}
Lemma \ref{lemma:eq_pert_3} implies that it exists $r$ such that 
\begin{align}
	\lambda_{i,N} - \frac{1}{M} \frac{\sigma^{2} c \lambda_i}{\lambda_i + \sigma^{2} \frac{1+c}{2}} + o\left(\frac{1}{M}\right)
	\notag
\end{align}
is solution of  \eqref{equation:zeros_gamma} contained in  $\Dcal_o(\lambda_i,r)$. It is however easy to check that
\begin{align}
	\phi_N \left(\lambda_{i,N} - \frac{1}{M} \frac{\sigma^{2} c \lambda_i}{\lambda_i + \sigma^{2} \frac{1+c}{2}} + o\left(\frac{1}{M}\right) \right) 
	=
	-\frac{\sigma^4 (1-c)^2}{2 \lambda_i} \left(1 - \frac{c}{2}\right) < 0.
	\notag
\end{align}
Therefore, the above extremum is negative, and its pre-image cannot be one the points $w_{q,N}^{-}, w_{q,N}^{+}$. 
Moreover, if $\lambda_i < \sigma^{2} \sqrt{c}$, then $h_1(\lambda_i) < 0$ and (\ref{equation:zeros_gamma}) has no extra real solution in  $\Dcal_o(\lambda_i,r)$. 
If  $\lambda_i > \sigma^{2} \sqrt{c}$, then $h_1(\lambda_i) > 0$, and the quantities
\begin{align}
	\lambda_{i,N} - \frac{1}{\sqrt{M}} \sqrt{h_1(\lambda_i)} + o\left(\frac{1}{\sqrt{M}}\right)
	\quad\text{and}\quad
	\lambda_{i,N} + \frac{1}{\sqrt{M}} \sqrt{h_1(\lambda_i)} + o\left(\frac{1}{\sqrt{M}}\right)
	\notag
\end{align}
are the 2 other real solutions of \eqref{equation:zeros_gamma} contained in $\Dcal_o(\lambda_i,r)$. 
After some algebra, we get that
\begin{align}
	\phi_N \left(\lambda_{i,N} - \frac{1}{\sqrt{M}} \sqrt{h_1(\lambda_i)} + o\left(\frac{1}{\sqrt{M}}\right) \right)
	&= 
	\frac{(\lambda_{i,N} + \sigma^2 c_N)(\lambda_{i,N} + \sigma^2)}{\lambda_{i,N}}   
	- \frac{1}{\sqrt{M}} \frac{2 \sqrt{h_1(\lambda_i)} (\lambda_i^2 - \sigma^4 c)}{\lambda_i^2}
	+ o\left(\frac{1}{\sqrt{M}}\right),
	\notag\\
	\phi_N\left(\lambda_{i,N} + \frac{1}{\sqrt{M}} \sqrt{h_1(\lambda_i)} + o\left(\frac{1}{\sqrt{M}}\right) \right)
	&= 
	\frac{(\lambda_{i,N} + \sigma^2 c_N)(\lambda_{i,N} + \sigma^2)}{\lambda_{i,N}} 
	+ \frac{1}{\sqrt{M}} \frac{2 \sqrt{h_1(\lambda_i)} (\lambda_i^2 - \sigma^4 c)}{\lambda_i^2}
	+ o\left(\frac{1}{\sqrt{M}}\right),
	\notag
\end{align}
are both positive.
It is easy to check that if $k \leq K_s$, then, $\sigma^2 \sqrt{c_N} < \lambda_{k,N}$ for $N$ large enough. The above discussion thus implies that 
$\Scal_N$ has $K_s+1$ clusters, and that for $k \leq K_s$, then
\begin{align}
	x_{1,N}^- &= \sigma^2 (1 - \sqrt{c_N})^2 + \mathcal{O}\left(\frac{1}{M}\right), 
	\notag\\
	x_{1,N}^+ &= \sigma^2 (1 + \sqrt{c_N})^2 + \mathcal{O}\left(\frac{1}{M}\right), 
	\notag\\
	x_{K_s+2-k,N}^- &= 
	\frac{(\lambda_{k,N} + \sigma^2 c_N)(\lambda_{k,N} + \sigma^2)}{\lambda_{k,N}} 
	- \frac{1}{\sqrt{M}} \frac{2 \sqrt{h_1(\lambda_k)} (\lambda_k^2 - \sigma^4 c)}{\lambda_k^2}
	+ o\left(\frac{1}{\sqrt{M}}\right),
	\notag\\
	x_{K_s+2-k,N}^+ &= 
	\frac{(\lambda_{k,N} + \sigma^2 c_N)(\lambda_{k,N} + \sigma^2)}{\lambda_{k,N}} 
	+ \frac{1}{\sqrt{M}} \frac{2 \sqrt{h_1(\lambda_k)} (\lambda_k^2 - \sigma^4 c)}{\lambda_k^2}
	+ o\left(\frac{1}{\sqrt{M}}\right).
	\notag
\end{align}
In order to complete the proof, we use theorem \ref{theorem:loc_eig}.
Let $k \in \{1,\ldots,K_s\}$. From the previous analysis, the eigenvalue $\lambda_{k,N}$ is the unique eigenvalue of ${\bf B}_N {\bf B}_N^{*}$ associated with interval 
$(w_{q,N}^-,w_{q,N}^+)$ with $q= K_s - k + 2$, for $N$ large enough. 
Moreover, the number of clusters of $\Scal_N$ is equal to $K_s + 1$ for $N$ large enough and the sequences $x_{q,N}^{-}$ and $x_{q,N}^{+}$ converge towards limits equal to 
$\sigma^{2}(1 - \sqrt{c})^{2}$ and  $\sigma^{2}(1 + \sqrt{c})^{2}$ for $q=1$, and both coincide with $\psi(\lambda_{K_s+2-q})$ for $q \geq 2$. 
This implies that for each $\epsilon > 0$, almost surely for $N$ large enough, then $\hat{\lambda}_{k,N} \in (\psi(\lambda_k)-\epsilon,\psi(\lambda_k)+\epsilon)$ for 
$k=1, \ldots, K_s$ and that $\hat{\lambda}_{k,N} \in (\sigma^{2}(1 - \sqrt{c})^{2}-\epsilon, \sigma^{2}(1 + \sqrt{c})^{2} + \epsilon)$ for $k > K_s$. 
This shows that $\hat{\lambda}_{k,N} \rightarrow \psi(\lambda_k)$ for $k=1, \ldots, K_s$. 

We now prove the convergence of $\hat{\lambda}_{k,N}$ to $\sigma^2 (1 + \sqrt{c})^2$ for $K_s < k \leq K$.  
Let $k_{\mathrm{max}} = K_s + 1$ (i.e the index of the largest eigenvalue associated with the first cluster $[x_{1,N}^-, x_{1,N}^+]$).
We have already shown  $\limsup_N \hat{\lambda}_{k_{\mathrm{max}},N} \leq \sigma^2 (1+\sqrt{c})^2$ almost surely. 
It remains to prove $\liminf_N \hat{\lambda}_{k_{\mathrm{max}},N} \geq \sigma^2 (1+\sqrt{c})^2$.
Assume the converse is true. Then it exists $\epsilon > 0$ such that $\liminf_N \hat{\lambda}_{k_{\mathrm{max}},N} < \sigma^2 (1+\sqrt{c})^2 - \epsilon$.
We can thus extract a subsequence $\hat{\lambda}_{k_{\mathrm{max}},\phi(N)}$ converging towards a limit less than $\sigma^2 (1+\sqrt{c})^2 - \epsilon$.
Let $\hat{\mu}_{\phi(N)}$ be the empirical spectral measure associated with matrix $\boldsymbol{\Sigma}_{\phi(N)}\boldsymbol{\Sigma}^*_{\phi(N)}$.
We deduce that 
\begin{align}
	\hat{\mu}_{\phi(N)}\left((\sigma^2 (1+\sqrt{c})^2 - \epsilon, \sigma^2 (1+\sqrt{c})^2]\right) = 0  \quad \text{a.s for all large N}.
	\label{eq:contradition}
\end{align}
Theorem \ref{theorem:conv_mh_N} implies that $\hat{\mu}_{\phi(N)}$ converges torwards the Marcenko-Pastur distribution, which contradicts \eqref{eq:contradition}.  
This proves that $\hat{\lambda}_{k_{\mathrm{max}},N} \to \sigma^2 (1+\sqrt{c})^2$ with probability one.
We can prove similarly that $\hat{\lambda}_{k,N} \to \sigma^2 (1+\sqrt{c})^2 \ a.s$ for $K_s+1 < k \leq K$.

	\subsection
	{
	        \texorpdfstring
	        {Characterization of $\Scal_N$ and limits of the eigenvalues if $\sigma^2 \sqrt{c} \in \{\lambda_1, \ldots, \lambda_K \}$}
	        {Characterization of the support and limits of eigenvalues 2}
	}

In this section, we handle the case where one the $(\lambda_k)_{k=1,\ldots,K}$, say $\lambda_j$ with $j \leq K$, is equal to $\sigma^2 \sqrt{c}$.
For this, we will use the Fan inequality (see \cite[Th.2]{fan1951maximum}). 
For a rectangular matrix $\A$, we will denote by $\kappa_k(\A)$ its k-th largest singular value.
With this notations, we have $\kappa_j(\B) = \sqrt{\lambda_j} = \sqrt{\sigma^2 \sqrt{c}}$.
We also denote by $\u_{j,N}$ and $\v_{j,N}$ the left and right singular vector of $\B_N$ associated with $\kappa_j(\B_N)$.
Fan inequality gives, for $\epsilon > 0$,
\begin{align}
	\kappa_j(\B_N + \sigma \W_N) & \leq \kappa_j(\B_N + \sigma \W_N + \epsilon\, \u_{j,N} \v_{j,N}^*) + \kappa_1(\epsilon\, \u_{j,N} \v_{j,N}^*),
	\notag
	\\
	\kappa_j(\B_N + \sigma \W_N + \epsilon\, \u_{j,N} \v_{j,N}^*) & \leq \kappa_j(\B_N + \sigma\W_N) + \kappa_1(\epsilon\, \u_{j,N} \v_{j,N}^*).
	\notag
\end{align}
From the results of the previous section, it is clear that, almost surely,
\begin{align}
	\kappa_j(\B_N + \sigma \W_N + \epsilon\, \u_{j,N} \v_{j,N}^*) = \sqrt{\psi\left(\left(\sqrt{\lambda_j} + \epsilon\right)^2\right)} + o(1).
	\notag
\end{align}
Therefore, we end up with
\begin{align}
	\sqrt{\psi\left(\left(\sqrt{\lambda_j} + \epsilon\right)^2\right)} - \epsilon
	\leq
	\liminf_N \, \kappa_j(\B_N + \sigma \W_N)
	\leq
	\limsup_N \, \kappa_j(\B_N + \sigma \W_N) 
	\leq 
	\sqrt{\psi\left(\left(\sqrt{\lambda_j} + \epsilon\right)^2\right)} + \epsilon.
	\notag
\end{align}
Since $\psi(\lambda) \to \sigma^2 (1+\sqrt{c})^2$ when $\lambda \to \sigma^2 \sqrt{c}$, this completes the results of theorem \ref{theorem:lim_spiked_eig}.

%%%%%%%%%%%%%%%%%%%%%%%%%%%%%%%%%%%%%%%%%%%%%%%%%%%%%%%%%%%%%%%%%%%%%%%%%%%%%%%%%%%%%%%%%%%%%%%%%%%%%%%%%%%%%%%%%%%%%%%%%%%%%%%%%%%%%%%%%%%%%%%%%%%%%%%%%%%%%%%%%%%
%
%
%
%
%	APPENDICES
%
%
%
%
%
%%%%%%%%%%%%%%%%%%%%%%%%%%%%%%%%%%%%%%%%%%%%%%%%%%%%%%%%%%%%%%%%%%%%%%%%%%%%%%%%%%%%%%%%%%%%%%%%%%%%%%%%%%%%%%%%%%%%%%%%%%%%%%%%%%%%%%%%%%%%%%%%%%%%%%%%%%%%%%%%%%%

\appendix

\section{\texorpdfstring{Proof of item \ref{item:re(b)} of theorem \ref{theorem:properties_m_N}}{Proof 1}}
\label{sec:proof-re(b)}

The proof is a  direct consequence of Section 2 of  \cite{dozier2007analysis}
(formulas (2.1) to (2.5) of \cite{dozier2007analysis}). We first recall that if we denote by 
$g(z)$ and $G(z)$ the terms defined for $z \in \mathbb{C}^{+}$ by
\begin{align}
	g(z) &= \frac{\sigma^2 c_N}{|1 + \sigma^2 c_N m_N(z)|^2} \frac{1}{M} \Tr \B_N \B_N^* \T_N(z) \T_N(z)^*,
	\notag\\
	G(z) &= \sigma^2 c_N \frac{1}{M} \Tr \T_N(z) \T_N(z)^*,
	\notag
\end{align}
where $\T_N(z) = \left[-z(1+\sigma^{2}c_{N}m_{N}(z))\mathbf{I}_{M}+\sigma^{2}(1-c_{N})\mathbf{I}_{M}+\frac{\mathbf{B}_{N}\mathbf{B}_{N}^{*}}{1+\sigma^{2}c_{N}m_{N}(z)} \right]^{-1}$, 
then, it is shown in \cite{dozier2007analysis} that 
\begin{align}
	\label{eq:Gg}
	0 < |z| G(z) < 1 - g(z)
\end{align}
for each $z \in \mathbb{C}^{+}$. 
If  $z_1 = \mathrm{Re}(z)$ and $z_2 = \mathrm{Im}(z)$, \eqref{eq:Gg} implies that $0 < 1 - g(z) -|z_1| G(z) \leq  1 - g(z) +z_1 G(z)$. 
It is shown in \cite{dozier2007analysis} that
\begin{align}
	\mathrm{Re}(1 + \sigma^{2} c_N m_N(z)) = \frac{1 + \sigma^{2}(1 - c_N) G(z) + \mathrm{Im}(1 + \sigma^{2} c_N m_N(z)) z_2 G(z)}{ 1 - g(z) +z_1 G(z)},
	\notag
\end{align}
for $z \in \mathbb{C}^{+}$. 
As $\mathrm{Im}(1 + \sigma^{2} c_N m_N(z)) = \sigma^{2} c_N \mathrm{Im}(m_N(z)) > 0$ on $\mathbb{C}^{+}$ (see item \ref{property:Im_Stieltjes_R+} of Property \ref{property:Stieltjes}), 
we get that 
\begin{align}
	\mathrm{Re}(1 + \sigma^{2} c_N m_N(z)) > \frac{1}{ 1 - g(z) +z_1 G(z)} > \frac{1}{1 + z_1 G(z)}.
	\notag
\end{align}
The inequality $|z_1| G(z) < 1$ implies that $\mathrm{Re}(1 + \sigma^{2} c_N m_N(z)) > \frac{1}{2}$ for each $z \in \mathbb{C}^{+}$. 
This also implies that  $\mathrm{Re}(1 + \sigma^{2} c_N m_N(x)) \geq \frac{1}{2}$ for $x \in \mathbb{R}$ if $c_N < 1$ and for $x \in \mathbb{R}^{*}$ if $c_N = 1$.

\section{\texorpdfstring{Proof of items \ref{item:left-continuity} and \ref{item:domination_der_w} of lemma \ref{lemma:property_w} when $q=1$}{Proof 2}}
\label{sec:proof-w'}

In order to prove these 2 statements, we study the behaviour of $w_N(x)$ and of $w_N^{'}(x)$ when $x \rightarrow 0, x < 0$ and $x \rightarrow 0, x > 0$. 

We first look at the limit for $x < 0$.
Lemmas \ref{lemma:property_phi_extrema} and \ref{lemma:property_w} imply that $w_N$ is the inverse of $\phi_N$ on interval $(-\infty,0)$. $w_N(x)$ is a continuous 
increasing function on $(-\infty, 0)$ upperbounded by $w_{1,N}^{-}$; therefore, $\lim_{x \rightarrow 0, x < 0} w_N(x)$ exists, and is less than  $w_{1,N}^{-}$. 
Taking the limit  when $x \rightarrow 0, x < 0$ from both sides of the equation $\phi_N(w_N(x)) = x$ for $x \in (-\infty, 0)$, and using the continuity of $\phi_N$ on $(-\infty, 0]$, 
we obtain immediately that $\phi_N(\lim_{x \rightarrow 0, x < 0}  w_N(x)) = 0$. This implies that $\lim_{x \rightarrow 0, x < 0}  w_N(x) = w_{1,N}^-$. 
This shows that $w_N$ is left continuous at $x=0$.
Since $w_N(x) = x (1 + \sigma^2 m_N(x))^2$, it follows that $1 + \sigma^2 m_N(x) = \mathcal{O}(|x|^{-1/2})$. 
As $w_N$ is continuously differentiable on $(-\infty, 0)$, we can differentiate the relation  $\phi_N(w_N(x)) = x$, and obtain that $\phi_N'(w_N(x)) w_N'(x) = 1$ for $x < 0$, 
or equivalently that $w_N'(x) = \frac{1}{\phi_N'(w_N(x))}$. In other words, it holds that 
\begin{align}
	w_N'(x) = \frac{1}{[1-\sigma^2 f_N(w_N(x))] [1-\sigma^2 f_N(w_N(x)) - 2 \sigma^2 w_N(x) f_N'(w_N(x))]}.
	\label{eq:expre-w'}
\end{align}
We observe that $1 - \sigma^{2} f_N(w_{1,N}^-) = 0$ so that 
\begin{align}
        \lim_{x\uparrow 0} 1-\sigma^2 f_N(w_N(x)) - 2 \sigma^2 w_N(x) f_N'(w_N(x))) = -2 \sigma^2 w_{1,N}^- f_N'(w_{1,N}^-) \neq 0
\end{align}
Moreover, \eqref{eq:expre-b-1} implies that 
\begin{align}
	\frac{1}{1 - \sigma^2 f_N(w_N(x))} = 1 + \sigma^2 m_N(x) \quad \text{for } x < 0,
	\notag
\end{align}
which proves that $(1 - \sigma^2 f_N(w_N(x)))^{-1} = \mathcal{O}((-x)^{-1/2})$. \eqref{eq:expre-w'} implies immediately that $w_N'(x) = \mathcal{O}\left(\frac{1}{\sqrt{-x}}\right)$.

We now study the behaviour of $w_N$ and $w_N^{'}$ when $x \rightarrow 0$, $x > 0$. We first study $\sqrt{x} m_N(x)$ for $x \rightarrow 0, x > 0$. For this, 
we introduce the function $\psi(\xi, y)$ defined by 
\begin{align}
	\psi(\xi,y) = 1 - \frac{1}{M} \Tr \left(\B_N \B_N^* \frac{\xi}{y + \sigma^2 \xi}  - \xi (y + \sigma^2 \xi) \right)^{-1}.
	\notag
\end{align}
The introduction of $\psi$ is based on the observation that eq. \eqref{definition:m_N} is equivalent to $\psi(\sqrt{x} m_N(x), \sqrt{x}) = 0$ for $x > 0$. 
We denote by $\xi_0$ the term $\xi_0 = i \sigma^{-2} \sqrt{|w_{1,N}^-|}$ and notice that $\psi(\xi_0, 0) = 0$.
It is easily checked that $\psi$ is holomorphic in a neighborhood of $(\xi_0, 0)$ and that $\frac{\partial \psi}{\partial \xi} (\xi_{0},0) \neq 0$. 
Therefore, from the implicit function theorem (the analytic version - see e.g Cartan \cite[Prop.6]{cartan1961theorie}), it exists a unique function $\xi(y)$, holomorphic in a neighborhood 
$\Vcal$ of $0$ satisfying $\psi(\xi(y), y) = 0$ for $y \in \Vcal$ and $\xi(0) = \xi_0$. 
As $\mathrm{Im}(\xi_0) >0$, it is clear that it exists a neighborhood $\Vcal^{'}$ of $0$ included in $\Vcal$ such that $\mathrm{Im}(\xi(y)) > 0$
for each $y \in \Vcal^{'}$. We claim that for $\sqrt{x} \in  \Vcal^{'} \cap \mathbb{R}^{+*}$, $\xi(\sqrt{x}) = \sqrt{x} m_N(x)$. For this, we 
notice that if $x \in (0, x_{1,N}^+)$, $m_N(x)$ is the unique solution of Eq. (\ref{definition:m_N}) for which $\Im(m_N(x)) > 0$. Indeed,  
from item \ref{item:w_eq_phi} of lemma \ref{lemma:property_w}, for $x \in  (0, x_{1,N}^+)$, $w_N(x)$ is the unique solution with positive imaginary part 
of  equation $\phi_N(w) = x$. 
But, $m_N(x)$ is solution of \eqref{definition:m_N} iff $w_N(x)$ is solution of $\phi_N(w) = x$. 
Moreover $m_N(x) \in \mathbb{C}^+$ iff $w_N(x)\in \mathbb{C}^+$, a property which is readily seen from the relation \eqref{eq:canonique-w}. 
The conclusion follows from the observation that $m_N(x)$ satisfies \eqref{definition:m_N} iff $\sqrt{x} m_N(x)$ satisfies $\psi(\sqrt{x} m_N(x), \sqrt{x}) = 0$. 
This in turn shows that for each $\sqrt{x} \in  \Vcal^{'} \cap \mathbb{R}^{+*}$, $\xi(\sqrt{x}) = \sqrt{x} m_N(x)$, 
or equivalently that $ \xi(y) = y m_N(y^{2})$ for $y \in  \Vcal^{'} \cap \mathbb{R}^{+*}$. 
As $\xi(y)$ is holomorphic in $\Vcal^{'}$, $\xi(y) = \xi_0 + o(1)$ and $\xi^{'}(y) = \xi_1 + o(1)$ for some coefficient $\xi_1$. 
Therefore, $y m_N(y^{2}) =  \xi_0 + o(1)$ and $2 y^{2} m_N^{'}(y^{2}) + m_N(y^{2}) = \xi_1 + o(1)$ for $y \in  \Vcal^{'} \cap \mathbb{R}^{+*}$, 
or equivalently $\sqrt{x} m_N(x) = \xi_0 + o(1)$ and $2 x m_N^{'}(x) + m_N(x) =  \xi_1 + o(1)$ for $x > 0$ small enough. As $w_N(x) = x (1 + \sigma^{2} m_N(x))^{2}$, 
we get that 
\begin{align}
	w_N^{'}(x) = \left(1 + \sigma^{2} m_N(x) \right) \left(1 + \sigma^{2} (m_N(x) +  2 x m_N^{'}(x)) \right).
	\notag
\end{align}
As $(m_N(x) +  2 x m_N^{'}(x))$ is a $\Ocal(1)$ term, and as $m_N(x) = \frac{\xi_0}{\sqrt{x}} + o(\frac{1}{\sqrt{x}})$, 
we obtain that $|w_N^{'}(x)| \leq \frac{C}{\sqrt{x}}$ for $x > 0$ small enough for some constant $C > 0$.

\section{\texorpdfstring{Proof of lemma \ref{lemma:eq_pert_3}}{Proof 3}}
\label{section:proof_lemma_eq_pert_3}

We begin by choosing $r > 0$ and $\epsilon_1 > 0$ such that $r < r_0$,  $z_{0,\epsilon} \in \mathcal{D}_{c}(z_{0},r)$ and 
$\mathcal{D}_{c}(z_{0, \epsilon},r) \subset \mathcal{D}_{0}(z_{0},r_0)$, for each 
$0 < \epsilon < \epsilon_1$.  Let $f_{\epsilon}(z) = (z - z_{0,\epsilon})^3 - \epsilon (z-z_{0,\epsilon}) h_{1,\epsilon}(z) + \epsilon^2 h_{2,\epsilon}(z)$ and
$g_{\epsilon}(z) = (z-z_{0,\epsilon})^3$. Moreover, define $\frac{K_i}{2} = \sup_{\mathcal{D}_{c}(z_{0},r)} |h_i(z)|$ (for $i=1,2$).

As $\sup_{z \in \Dcal_{o}(z_0, r_0)} |\chi_{i,\epsilon}(z)| = o(1)$, it exists $\epsilon_2 \leq \epsilon_1$ such that $\sup_{\mathcal{D}_{c}(z_{0},r)} |h_{i,\epsilon}(z)| \leq K_i$ 
(for $i=1,2$)
for each $\epsilon \leq \epsilon_2$.
For $z \in \mathcal{D}_{c}(z_{0},r)$, it holds that
\begin{align}
	\left|f_{\epsilon}(z) -  g_{\epsilon}(z)\right| 
	&\leq 
	\epsilon \left|z-z_{0,\epsilon}\right| \left|h_{1,\epsilon}(z)\right|
	+ \epsilon^2 \left|h_{2,\epsilon}(z)\right|.
	\notag
\end{align}
As $z_{0,\epsilon} -z_0= o(1)$, it exists $\epsilon_3 \leq \epsilon_2$ such that, for each $\epsilon \leq \epsilon_3$, $|z-z_{0,\epsilon}| < 2r$ on  $\mathcal{D}_{c}(z_{0},r)$. Hence, 
for each $\epsilon \leq \epsilon_3$, it holds that $\left|f_{\epsilon}(z) -  g_{\epsilon}(z)\right|\leq 2 \epsilon r K_1 + \epsilon^2 K_2$ on $ \mathcal{D}_{c}(z_{0},r)$. 
We now restrict $z$ to $\mathcal{C}(z_0,r)$, the boundary of $\mathcal{D}_{c}(z_{0},r)$. It exists $\epsilon_4 \leq \epsilon_3$ for which 
$2 \epsilon r K_1 + \epsilon^2 K_2 < \frac{r^3}{2} < r^{3} = |z - z_0|^{3}$ holds on $\mathcal{C}(z_0,r)$ for each $\epsilon \leq \epsilon_4$. 
Therefore, $\forall z \in \mathcal{C}(z_0,r)$, we have $|f_{\epsilon}(z) - g_{\epsilon}(z)| < |g_{\epsilon}(z)|$ for  $\epsilon \leq \epsilon_4$. 
It follows from Rouché's theorem that these values of $\epsilon$,  then $f_{\epsilon}$ and $g_{\epsilon}$ have the same number of zeros inside $\mathcal{D}_{o}(z_{0},r)$.
Thus, for $\epsilon \leq \epsilon_4$, the equation
\begin{align}
	(z - z_{0,\epsilon})^3 - \epsilon (z-z_{0,\epsilon}) h_{1,\epsilon}(z) + \epsilon^2 h_{2,\epsilon}(z) = 0
	\label{equation:eq_pert_1}
\end{align}
has three solutions in $\mathcal{D}_{o}(z_{0},r)$.
Using the the same procedure to functions 
$f_{\epsilon}(z) = (z - z_{0,\epsilon})^2 - \epsilon h_{1,\epsilon}(z)$ and $g_{\epsilon}(z) = (z-z_{0,\epsilon})^2$,
we deduce that if $\epsilon \leq \epsilon_5 \leq \epsilon_4$, the equation
\begin{align}
	(z - z_{0,\epsilon})^2 - \epsilon h_{1,\epsilon}(z) = 0
	\label{equation:eq_pert_2}
\end{align}	
has two solutions  $\hat{z}_{\epsilon}^-, \hat{z}_{\epsilon}^+$ in $\mathcal{D}_{o}(z_{0},r)$. 
We clearly have $|z_{0,\epsilon} - \hat{z}_{\epsilon}^-| =  \mathcal{O}(\epsilon^{1/2})$ and $|z_{0} - \hat{z}_{\epsilon}^-| = o(1)$. 
Therefore, $h_{1,\epsilon}(\hat{z}_{\epsilon}^-) - h_1(z_0) = o(1)$. 
As $h_1({z_0}) \neq 0$, it exists $\epsilon_6 \leq \epsilon_5$ and a neighborhood of $h_1({z_0})$, containing $h_{1,\epsilon}(\hat{z}_{\epsilon}^-), h_{1,\epsilon}(z_{0})$ 
for each $\epsilon \leq \epsilon_6$, in which  a suitable branch of the square-root $\sqrt{.}$ is analytic. 
We assume that solution $\hat{z}_{\epsilon}^-$ is given by $z_{0,\epsilon} - \hat{z}_{\epsilon}^- = -\sqrt{\epsilon} \sqrt{h_{1,\epsilon}(\hat{z}_{\epsilon}^-)}$.
As $|h_{1}(z_{0}) - h_{1,\epsilon}(\hat{z}_{\epsilon}^-)| = o(1)$, we have $z_{0,\epsilon} - \hat{z}_{\epsilon}^-= -\sqrt{\epsilon} \sqrt{h_{1}(z_0)} + o(\sqrt{\epsilon})$.
We obtain similarly that  $z_{0,\epsilon} - \hat{z}_{\epsilon}^+ = \sqrt{\epsilon} \sqrt{h_{1}(z_0)}  + o(\sqrt{\epsilon})$.

Considering again $\hat{z}_{\epsilon}^-$, it follows that it exists $\epsilon_7 \leq \epsilon_6$ such that for each $\epsilon \leq \epsilon_7$, it holds that 
\begin{align}
	\left|z_{0,\epsilon} - \hat{z}_{\epsilon}^-\right| > \frac{\sqrt{\epsilon} \sqrt{h_1(z_0)}}{2} > \sqrt{\epsilon}\sqrt{r'},
	\label{equation:boundinf1}
\end{align}
with $r' < \frac{|h_1(z_0)|}{4}$. For $\epsilon \leq \epsilon_8 \leq \epsilon_7$, we have $\sqrt{\epsilon r'} < r$ and for $z \in \mathcal{D}_{c}(z_{0,\epsilon},\sqrt{\epsilon r'})$, 
we get
\begin{align}
	\left|(z - z_{0,\epsilon})^2 - \epsilon h_{1,\epsilon}(z)\right| > 
	\epsilon |h_{1,\epsilon}(z)| - |z - z_{0,\epsilon}|^{2} >  \epsilon \left( |h_{1,\epsilon}(z)| - r^{'} \right).
	\notag
\end{align}
It is easy to check that for each $\epsilon \leq \epsilon_9 \leq \epsilon_8$, then  $|h_{1,\epsilon}(z)| > \frac{|h_{1}(z_0)|}{2}$ for 
$z \in \mathcal{D}_c(z_{0,\epsilon},\sqrt{\epsilon r'})$. Therefore, 
\begin{align}
	\left|(z - z_{0,\epsilon})^2 - \epsilon h_{1,\epsilon}(z)\right| 
	> \epsilon \left(\frac{|h_{1}(z_0)|}{2} - r'\right)
	> \epsilon r'.
	\label{equation:boundinf2}
\end{align}
The inequalities \eqref{equation:boundinf1} and \eqref{equation:boundinf2} prove that in $\mathcal{D}_c(z_{0,\epsilon},\sqrt{\epsilon r'})$, 
the equation \eqref{equation:eq_pert_2} has no solution and that the equation $(z - z_{0,\epsilon})^3 - \epsilon(z - z_{0,\epsilon}) h_{1,\epsilon}(z) = 0$ has only one solution there.

We  now study the number of solutions in $\mathcal{D}_c(z_{0,\epsilon},\sqrt{\epsilon r'})$ of the equation \eqref{equation:eq_pert_1}. Consider
\begin{align}
	f_{\epsilon}(z) &= (z - z_{0,\epsilon})^3 - \epsilon(z - z_{0,\epsilon}) h_{1,\epsilon}(z) + \epsilon^2 h_{2,\epsilon}(z),
	\notag\\
	g_{\epsilon}(z) &= (z - z_{0,\epsilon})^3 - \epsilon(z - z_{0,\epsilon}) h_{1,\epsilon}(z). \notag
\end{align}
We have $|f_{\epsilon}(z) - g_{\epsilon}(z)| = \epsilon^2 |h_{2,\epsilon}(z)|$. We consider $z\in \mathcal{C}(z_{0,\epsilon},\sqrt{\epsilon r'})$. 
From \eqref{equation:boundinf2}, $|g_{\epsilon}(z)| > (\epsilon r')^{3/2}$. Therefore, for each $\epsilon \leq \epsilon_{10} \leq \epsilon_9$, it holds that
$|g_{\epsilon}(z)| > \epsilon^2 |h_{2,\epsilon}(z)| = |f_{\epsilon}(z) - g_{\epsilon}(z)|$. Thus, from Rouché's theorem, the equation \eqref{equation:eq_pert_1}
has only one solution in $\mathcal{D}_o(z_{0,\epsilon},\sqrt{\epsilon r'})$, denoted by $z_{\epsilon}$.
To obtain $z_{\epsilon}$, we write
\begin{align}
	z_{\epsilon} - z_{0,\epsilon} = \frac{- \epsilon^2 h_{2,\epsilon}(z_{\epsilon})}{(z_{\epsilon} - z_{0,\epsilon})^2 - \epsilon h_{1,\epsilon}(z_{\epsilon})}.
	\notag
\end{align}
Since $|(z - z_{0,\epsilon})^2 - \epsilon h_{1,\epsilon}(z)| > \epsilon r^{'}$ on $\mathcal{D}_c(z_{0,\epsilon},\sqrt{\epsilon r'})$ (see \eqref{equation:boundinf2}), 
we get that
\begin{align}
	\left|z_{\epsilon} - z_{0,\epsilon}\right| 
	\leq 
	\frac{\epsilon^2 K_2}{\epsilon r'}
	= \mathcal{O}\left(\epsilon\right).
	\notag
\end{align}
But from equation \eqref{equation:eq_pert_1}, we also have
$\epsilon (z_{\epsilon}-z_{0,\epsilon}) h_{1,\epsilon}(z_{\epsilon}) = (z_{\epsilon} - z_{0,\epsilon})^3  + \epsilon^2 h_{2,\epsilon}(z_{\epsilon})$ which leads to
\begin{align}
	z_{\epsilon}-z_{0,\epsilon} = 
	\epsilon \frac{h_{2,\epsilon}(z_{\epsilon})}{h_{1,\epsilon}(z_{\epsilon})} + \frac{(z_{\epsilon}-z_{0,\epsilon})^3}{\epsilon h_{1,\epsilon}(z_{\epsilon})}.
	\notag
\end{align}
It is clear that
\begin{align}
	\frac{h_{2,\epsilon}(z_{\epsilon})}{h_{1,\epsilon}(z_{\epsilon})}
	-
	\frac{h_{2}(z_{0})}{h_{1}(z_{0})}
	=
	o(1),
	\notag
\end{align}
so that
\begin{align}
	z_{\epsilon}-z_{0,\epsilon} = 
	\epsilon \frac{h_{2}(z_{0})}{h_{1}(z_{0})} 
	+ o(\epsilon).
	\notag
\end{align}
We now evaluate the two remaining solutions of \eqref{equation:eq_pert_1} located in the set
$\mathcal{D}_o(z_0,r)\backslash \mathcal{D}_o(z_{0,\epsilon}, \sqrt{\epsilon r'})$, denoted $z_{\epsilon}^-$, $z_{\epsilon}^+$.
As $|z_{\epsilon}^- - z_{0,\epsilon}| > \sqrt{r' \epsilon}$, we can write
\begin{align}
	\left(z_{\epsilon}^- - z_{0,\epsilon}\right)^2 = 
	\epsilon h_{1,\epsilon}(z_{\epsilon}^-) - \epsilon^2 \frac{h_{2,\epsilon}(z_{\epsilon}^-)}{z_{\epsilon}^- - z_{0,\epsilon}} 
	\label{eq:utile}
\end{align}
This implies that $|z_{\epsilon}^- - z_{0,\epsilon}| = \Ocal(\sqrt{\epsilon})$ and that  $|z_{\epsilon}^- - z_{0}| = o(1)$. 
Taking a suitable branch of the square root, \eqref{eq:utile} implies that
\begin{align}
	z_{\epsilon}^- - z_{0,\epsilon}	= - \sqrt{\epsilon h_{1,\epsilon}(z_{\epsilon}^-)} + o(\sqrt{\epsilon}) = -  \sqrt{\epsilon h_{1}(z_0)} + o(\sqrt{\epsilon}).
	\notag
\end{align}
We obtain similarly that $z_{\epsilon}^+ - z_{0,\epsilon} =  \sqrt{\epsilon h_{1}(z_0)} + o(\sqrt{\epsilon})$. 

We finally verify that if $z_0$ and $z_{0,\epsilon}$ belong to $\mathbb{R}$ for each $\epsilon$, and that $h_i(z)$ and $h_{i,\epsilon}(z)$  belong to $\mathbb{R}$ for each $\epsilon$ 
if $z \in \mathbb{R}$ for $i=1,2$, then $z_{\epsilon}$ is real while $z_{\epsilon}^-, z_{\epsilon}^+$ are real if $h_1(z_0) > 0$. 

If $z_{\epsilon}$ is not real, it is clear that $z_{\epsilon}^{*}$ is also solution of \eqref{equation:eq_pert_1} because functions $h_{i,\epsilon}$ 
verifies $(h_{i,\epsilon}(z))^{*} = h_{i,\epsilon}(z^{*})$. 
As $|z_{\epsilon}^{*} - z_{0,\epsilon}| = |z_{\epsilon} - z_{0,\epsilon}| = \Ocal(\epsilon)$, and that \eqref{equation:eq_pert_1}  has a unique solution in the disk 
$\Dcal_o(z_{0,\epsilon}, \sqrt{\epsilon r^{'}})$, this implies that $z_{\epsilon}^{*} = z_{\epsilon}$. 
On the other hand, assume that $h_1(z_0) > 0$ and the  $z_{\epsilon}^-, z_{\epsilon}^+$ are non-real. 
Then, $z_{\epsilon}^{+ \ *}$ and $z_{\epsilon}^{- \ *}$ are also solution of \eqref{equation:eq_pert_2}. 
Since equation \eqref{equation:eq_pert_2} has only two solutions outside the disk $\Dcal_o(z_{0,\epsilon}, \sqrt{\epsilon r^{'}})$, it follows that $\hat{z}_{\epsilon}^{+}$ and 
$\hat{z}_{\epsilon}^{-}$ are complex conjuguate. 
But as their real parts have opposite sign for $\epsilon$ small enough, this leads to a contradiction.
Therefore $\hat{z}_{\epsilon}^{+}$ and $\hat{z}_{\epsilon}^{-}$ are real. We finally note that if $h_1(z_0) < 0$, then $\hat{z}_{\epsilon}^{+}$ and $\hat{z}_{\epsilon}^{-}$  are non real.

\end{document}